\documentclass[11pt]{reportbis}
\usepackage{amsmath}
\usepackage{amsfonts}
\usepackage{epsf}

\newcommand{\bet}{\noalign{\vskip6pt plus 3pt minus 1pt}}

\newcommand{\betless}{\noalign{\vskip3pt plus 3pt minus 1pt}}

\renewcommand{\lor }{\longrightarrow}
\renewcommand{\O }{\Omega }

\newtheorem{Theorem}{Theorem}
\newtheorem{lema}{Lemma}
\newcounter{remark}
\def\theremark {\arabic{remark}}
\newenvironment{remark}{\refstepcounter{remark}\par\noindent{\bf Remark\ \theremark}\ }{\par}
\newtheorem{Proof}{Proof}

\newenvironment{proof}{\begin{Proof}\rm}{\hfill $\Box$ \end{Proof}}
\begin{document}
\title{Optimal error bounds for two-grid schemes applied to the Navier-Stokes equations}
\author{Javier de Frutos\thanks{Departamento de Matem\'{a}tica Aplicada,
Universidad de Valladolid. Spain. Research supported by Spanish MICINN under grant MTM2010-14919,
and by JCyL grant VA001A10-1 (frutos@mac.uva.es)} \and Bosco Garc\'{\i}a-Archilla\thanks{Departamento de
Matem\'atica Aplicada II, Universidad de Sevilla, Sevilla, Spain. Research supported by Spanish
MICINN under grant MTM2009-07849 (bosco@esi.us.es)}
  \and Julia Novo\thanks{Departamento de
Matem\'aticas, Universidad Aut\'onoma de Madrid, Instituto de Ciencias Matem\'aticas
CSIC-UAM-UC3M-UCM, Spain. Research supported by Spanish MICINN under grant MTM2010-14919
(julia.novo@uam.es)}} \maketitle \abstract{We consider two-grid mixed-finite element schemes for
the spatial discretization of the incompressible Navier-Stokes equations. A standard mixed-finite
element method is applied over the coarse grid to approximate the nonlinear Navier-Stokes equations
while a linear evolutionary problem is solved over the fine grid. The previously computed Galerkin
approximation to the velocity is used to linearize the convective term. For the analysis we take
into account the lack of regularity of the solutions of the Navier-Stokes equations at the initial
time in the absence of nonlocal compatibility conditions of the data. Optimal error bounds are
obtained.}


\section{Introduction}
In this paper we study two-grid mixed finite-element (MFE) methods for the spatial discretization
of the incompressible Navier--Stokes equations
\begin{eqnarray}
\label{onetwo}
u_t -\Delta u + (u\cdot\nabla)u + \nabla p &=& f,\\
{\rm div}(u)&=&0,\nonumber
\end{eqnarray}
in a bounded domain $\Omega\subset {\mathbb R}^d$ ($d=2,3$) with a smooth
boundary subject to homogeneous Dirichlet boundary conditions $u=0$
on~$\partial\Omega$. In~(\ref{onetwo}), $u$ is the velocity
field, $p$ the pressure, and~$f$ a given force field. As in \cite{heyran0}, \cite{heyran2}, \cite{HR-IV}
we assume that
the fluid density and viscosity have been normalized by an adequate change
of scale in space and time. We approximate the solution $u$ and $p$
corresponding to a given
initial condition
\begin{equation}\label{ic}
u(\cdot,0)=u_0.
\end{equation}
Two-grid methods are a well established technique for nonlinear steady problems, see \cite{Xu1}, \cite{Xu}.
The main idea in a two-level method
 involves the discretization of the equations over two meshes of different size.
  A nonlinear system over the coarse mesh is solved in the first step of the method.
  In a second step, a linearized equation based on the approximation over the coarse
  mesh is solved on the fine mesh.
In \cite{Lay-Tob}, \cite{Lay-Len} several two-level methods are considered to approximate the
steady Navier-Stokes equations. In these papers, depending on the algorithm, the second step is
based on the solution of a discrete Stokes problem, a linear Oseen problem or one step of the
Newton method over the fine mesh with the coarse mesh approximation as initial guess.

Several two-level or two-grid schemes have also been considered in the literature to
approximate the evolutionary nonlinear Navier-Stokes equations (\ref{onetwo})-(\ref{ic}).
Again, two approximations to the velocity (and correspondingly two approximations to the pressure),
are computed. One of them is defined by a discretization of the nonlinear equations over a coarse
mesh, $u_H$,  and another one, $\tilde u_h$, is defined by an appropriate linearization
over a fine mesh. Different classes of algorithms can be seen as two level methods.
In particular, although they were originally developed from different ideas, the so called nonlinear
Galerkin methods, postprocessed and dynamical postprocessed methods, fall into this category.

Postprocessed Galerkin methods were first introduced for Fourier spectral methods
in \cite{Bosco-Julia-Titi}, \cite{Bosco-Julia-Titi-2} (see also \cite{Margolin_titi_wynne}) and later extended
to finite element methods
in \cite{bbj}, \cite{bjj}, \cite{jbj_regularity}, \cite{jbj_fully}. In all these works the main idea is the following:
one first compute
the standard Galerkin approximation to the velocity and pressure over a coarse mesh $(u_H,p_H)$ of size $H$
and then compute
the postprocessed approximation in a finer mesh
 at selected times in which one wants to obtain an improved approximation. More precisely, the postprocessed
approximation $(\tilde u_h,\tilde p_h)$
computed at a given time $t^*$
is an approximation in a mesh of size $h\ll H$ to the following (steady) Stokes problem:
\begin{equation}
\begin{array}{rcl}
\left.\begin{array}{r@{}}
-\Delta \tilde u +\nabla \tilde p=f-\frac{d}{dt}u_H({t^*})
-(u_H({t^*})\cdot\nabla) u_H({t^*})
\\
\betless
{\rm div}(\tilde u)=0
\end{array}\right\}&\quad \hbox{\rm in~$\Omega$},
\\
\bet
&\hspace*{-41pt}\tilde u=0, \quad \hbox{\,\rm on~$\partial\Omega$}.
\end{array}
\label{eq:stokes}
\end{equation}
Here, $u_H(t)$, $t\in (0,T]$, is the standard MFE approximation computed in the coarse mesh in a time interval
$(0,T]$ and $t^*\in (0,T]$. 
Note that the computation of $(u_H(t), p_H(t))$, $t\in (0,T]$,
is completely independent of the computation of
$(\tilde{u}_h(t^*),\tilde{p}_h(t^*))$ in the fine mesh.
The postprocessed approximation improves the rate of convergence of the standard Galerkin approximation
over the coarse mesh in the following sense.
If the rate of convergence of the Galerkin approximation to the velocity in the $L^2(\Omega)^d$ ($j=0$)
or $H^1(\Omega)^d$ ($j=1$) norm is $O(H^{r-j})$ then
the rate of convergence of the postprocessed approximation to the velocity is
$O(H^{r+1-j}\left|\log(H)\right|)$+$O(h^{r-j})$. Analogous results
are obtained for the pressure. For first order mixed finite element methods the improvement in the rate of convergence
of the velocity is only achieved in the $H^1(\Omega)^d$ norm, \cite{bjj}.
Then, if one wants to achieve the optimal accuracy of the fine
level in the $H^1(\Omega)^d$ norm, one can first compute the Galerkin approximation
on a coarse mesh of size $H=h^{(r-1)/r}$ and
then compute the postprocessed approximation over the fine mesh of size $h$ at the desired time levels.
For example, one should take $H=h^{1/2}$ and~$H=h^{2/3}$
for linear and quadratic mixed finite elements, respectively. It can be expected
that the computational cost of the postprocessed approximation is smaller than that of
the Galerkin approximation on the same fine mesh, since in the former method
the time evolution is done on the coarse mesh, and only at selected time levels are
computations done on the fine mesh. This has been confirmed by the numerical
experiments in~\cite{bbj} (see also~\cite{novo3} and \cite{Bosco-Julia-Titi})

In \cite{Margolin_titi_wynne} a related
algorithm, the so-called dynamical postprocessing, is introduced for the Fourier case.
In this algorithm, the standard Galerkin approximation,  $(u_H,p_H)$,
is computed over a coarse mesh in the first level, as before. For the second level an approximation to
a linear evolutionary problem, instead of the steady problem (\ref{eq:stokes}), is computed. More precisely,
the dynamical postprocessing involves the approximation, at each time step, over a mesh of size $h\ll H$ of the problem:
\begin{equation}
\begin{array}{rcl}
\left.\begin{array}{r@{}}
\frac{d}{dt}\tilde u-\Delta \tilde u +\nabla \tilde p=f
-(u_H\cdot\nabla) u_H
\\
\betless
{\rm div}(\tilde u)=0
\end{array}\right\}&\qquad \hbox{\rm in~$\Omega$},
\\
\bet
&\hspace*{-41pt}\tilde u=0, \quad \hbox{~~~~~~~~~~\,\rm on~$\partial\Omega$}.
\end{array}
\label{eq:dyn}
\end{equation}
Note that in the dynamical postprocessing, the computation of
$(u_H(t), p_H(t))$ and $(\tilde{u}_h(t),\tilde{p}_h(t))$, $t\in [0,T]$,
is coupled.
The rate of convergence of the dynamical postprocessing scheme is proved in \cite{Margolin_titi_wynne}
to be the same as the rate of
convergence of the standard postprocessing. In the case of highly oscillatory solutions the dynamical
algorithm is shown to be more
efficient than the standard postprocessing in some one dimensional examples. The dynamical postprocessing
method is also considered in \cite{Olshanskii},
named now as two-level method, in the case of mixed finite elements.
In~\cite{Olshanskii}, the author treats the fully discrete case integrating in time
with the backward Euler method. A similar two-level scheme is also considered and analyzed in \cite{He} where
the author uses first order mixed finite elements in space,  Crank-Nicolson extrapolation for the
time integration over the coarse mesh and the backward Euler method for the time integration over the fine mesh.

The so-called nonlinear Galerkin methods are also two-level methods that have been used to compute
approximations to (\ref{onetwo})-(\ref{ic}).
They were first introduced for Fourier spectral methods \cite{Devulder_marion_titi},
\cite{Marion-Temam}, and later extended to mixed finite element
methods in \cite{Ait_etal}. In this work the authors obtain the rate of convergence of the nonlinear Galerkin method
in the case of first order elements. The rate of convergence is the same one of the postprocessed method.
The main difference between the nonlinear Galerkin methods and the postprocessed or two-grid methods
is that in the former  the approximation on the coarse mesh
takes into account the influence of the fine mesh, whereas in the latter
it is just the standard
Galerkin method (i.e., computed independently of the fine mesh).

In this paper we analyze two two-grid algorithms in the context of spatial mixed finite element
discretizations to approximate
the solutions of  (\ref{onetwo})-(\ref{ic}). The two algorithms we consider
 are very similar to
the dynamical postprocessing method. The difference is the treatment of the nonlinearity in the second level. In the
dynamical postprocessing method the nonlinear convective term of the fine level is approximated by
the data $(u_H\cdot \nabla)u_H$ (see  the
right-hand-side of (\ref{eq:dyn})).  In the two algorithms
we consider in the present paper, the approximation to the velocity of the coarse mesh $u_H$ is used to linearize
the nonlinear convective term of the
 fine level. In the first algorithm, the linearized convective term of the fine level is $(u_H\cdot \nabla)\tilde u_h$.
In the second algorithm $u_H$ is regarded as an initial guess to perform one Newton step in the fine level.
For the spatial discretization we consider mixed finite elements of
first, second and third order. More precisely, we consider the mini-element and the quadratic and cubic Hood-Taylor
elements. The analysis for other mixed finite elements of the same order is similar.
As in \cite{heyran2}, \cite{jbj_regularity} due to the lack of regularity at $t=0$ of the solution
of (\ref{onetwo})-(\ref{ic}) no better than $O(H^5)$ error bounds
can be expected. For this reason we do not analyze higher than cubic finite element discretizations.
For the temporal discretization we
use the backward Euler method or the two-step backward differentiation formula.
The analysis of the fully discrete methods is similar to the one appeared in \cite{jbj_fully}
and it is only briefly indicated in this paper.

 This is
not the first time these two algorithms have been considered. The first algorithm was introduced
in \cite{Girault-Lions}, where
the authors analyze the semi-discrete in space case for first order finite elements.
In \cite{Abboud_Saya} the authors
 extend this analysis to the fully discrete case and in \cite{Abboud_etal} the second order Hood-Taylor
 finite element is used
 for the spatial discretization and the two-step backward differentiation formula for the time integration.
 In \cite{Hou_etal} the second algorithm is analyzed for the Fourier
spectral case while in \cite{Liu_Hou} the analysis is extended to the case of first order mixed finite
elements considering the fully discrete
case coupled with the Crank-Nicolson scheme for the time integration.
As opposed to the above mentioned works on the same methods, in the present paper
we take into account the lack of regularity suffered by the solutions of the Navier-Stokes
  equations at the initial time. Then, for the analysis in the present paper we do not assume
more than second-order spatial derivatives bounded in $L^2$ up to initial time $t=0$,
since demanding further regularity requires the data to satisfy nonlocal
compatibility conditions unlikely to be fulfilled in practical situations
\cite{heyran0}, \cite{heyran2}.  This is the first time these methods are analyzed under realistic regularity assumptions.
Also, this
is the first time the cubic case is considered and the first time the quadratic case is considered for the second method.

There are
some other improvements with respect to previous works. In \cite{Abboud_etal} the authors get an error bound of
order $O(H^3+h^2+(\Delta t)^2)$ for the fine approximation to the velocity $\tilde u_h$ in the $H^1(\Omega)^d$ norm
whenever the following inequality is satisfied
 $
 \alpha_1 H^3\le (\Delta t)^2\le \alpha_2 H^3,
 $
  $\alpha_1$ and  $\alpha_2$ being constants independent of $H$ and $\Delta t$. With the technique of this paper
  an error bound of order
  $O(|\log(h)||\log(H)|H^4+h^2+(\Delta t)^2)$ for the same fully discrete method in the $H^1(\Omega)^d$ norm
  can be obtained for $H$ and $\Delta t$ independently chosen.
   With the new error bound obtained in this paper one can achieve the rate of convergence of the fine mesh in
  the $H^1(\Omega)^d$ norm by taking $H=h^{1/2}$ instead of $H=h^{2/3}$. This fact improves the efficiency of
  the method compared with
  the (same order) standard Galerkin method over the fine mesh. Also, the authors of \cite{Abboud_etal} remark
  that they have observed the same rate of convergence for the two-grid method with $H=h^{1/2}$ and $H=h^{2/3}$ in
  the numerical tests they have carried out, which supports the improved rate of convergence we obtain in this paper.
  We want to remark that in all the numerical
   experiments of \cite{Hou_etal}, \cite{Liu_Hou} and \cite{Abboud_etal} the two-grid algorithms improve the efficiency
   of the standard Galerkin method in the sense that a given error can be achieved with less computational cost with the
   new algorithms than with the standard Galerkin method. In \cite{Hou_etal} a comparison in the Fourier case between the
   standard postprocessing, the dynamical postprocessing and the second two-grid algorithm is also included.
   Although the computational
   cost of the two-grid approximation over the fine mesh is bigger than that of the postprocessed approximations, the
   two-grid algorithm produces smaller errors in the case of moderate to high Reynolds numbers. Finally,
comparing the two algorithms we analyze in this paper we remark that with the second algorithm better error bounds are
obtained in terms of $H$. Although this fact could make the choice of the second algorithm preferable for computations,
it turns out in practice to be rather inefficient to solve the linear systems accurately.
For this reason, some
authors suggest solving instead an Oseen problem leading then to the first algorithm, see \cite{Lay-Tob}.

The rest of the paper is as follows. In Section 2 we introduce some preliminaries and notations. In Section 3 we carry out the error
analysis of the first two-grid algorithm in the semi-discrete in space case. In Section 4 we consider the analysis of the
second two-grid algorithm in the semi-discrete in space case.  Finally, in Section 5 we consider the fully discrete case integrating in time with the backward Euler method or
the two-step backward differentiation formula.


\section{Preliminaries and notations}
We will assume that
$\Omega$ is a bounded domain in~${\mathbb{R}}^{d},\, d=2,3$, not
necessarily convex and smooth enough. When dealing with linear elements
($r=2$ below) $\Omega$ may also be a convex polygonal or polyhedral domain.
We will consider
the Hilbert spaces
\begin{align*}
 H&=\{ u \in \big(L^{2}(\Omega))^d \mid\mbox{div}(u)=0, \, u\cdot n_{|_{\partial \Omega}}=0 \},\\
\bet
V&=\{
u \in \big(H^{1}_{0}(\Omega))^d \mid \mbox{div}(u)=0 \},
\end{align*}
endowed
with the inner product of $L^{2}(\Omega)^{d}$ and $H^{1}_{0}(\Omega)^{d}$,
respectively. For $l\ge 0$ integer and $1\le q\le \infty$, we consider the
standard Sobolev spaces $W^{l,q}(\Omega)^d$ of functions with
derivatives up to order $l$ in $L^q(\Omega)$, and
$H^l(\Omega)^d=W^{l,2}(\Omega)^d$. We will denote by $\|\cdot \|_l$ the norm in $H^l(\Omega)^d$,
and~$\|\cdot\|_{-l}$ will represent
the norm of its dual space. We consider also the quotient spaces
$H^l(\Omega)/{\mathbb R}$ with norm $\left\| p\right\|_{H^l/{\mathbb
R}}= \inf\{ \left\| p+c\right\|_l\mid c\in {\mathbb R}\}$.

Let us recall the following Sobolev's imbedding inequalities \cite{Adams}: For
$q \in [1, \infty)$, there exists a constant $C=C(\Omega, q)$ such
that
\begin{equation}\label{sob1}
\|v\|_{L^{q'}(\Omega)^{d}} \le C \| v\|_{W^{s,q}(\Omega)^{d}}, \,\,\quad
\frac{1}{q'}
\ge \frac{1}{q}-\frac{s}{d}>0,\quad q<\infty, \quad v \in
W^{s,q}(\Omega)^{d}.
\end{equation}
For $q'=\infty$, (\ref{sob1}) holds with $\frac{1}{q}<
\frac{s}{d}$.

Let $\Pi:
L^2(\Omega)^d \lor H$ be the $L^2(\Omega)^d $ projection onto $H$.
We denote by $A$ the Stokes
operator on $\Omega$:
$$
A: \mathcal{D}(A)\subset H \lor H, \quad \, A=-\Pi\Delta , \quad
\mathcal{D}(A)=H^{2}(\Omega)^{d} \cap V.
$$
Applying Leray's projector to
(\ref{onetwo}), the equations can be written in the form
\begin{eqnarray*}
u_t + A u +B(u,u) =\Pi f \qquad \mbox{ in } \O,
\end{eqnarray*}
where $B(u,v)=\Pi (u\cdot \nabla) v$ for $u$, $v$ in $H_0^1(\Omega)^d$.

We shall use the trilinear form $b(\cdot,\cdot,\cdot)$ defined by
$$
b(u,v,w)=(F(u,v),w)\quad\forall u,v,w\in H_0^1(\Omega)^d,
$$
where
$$
F(u,v)=(u\cdot \nabla) v +\frac{1}{2}(\nabla\cdot u)v\quad\forall u,v\in H_0^1(\Omega)^d.
$$
It is straightforward to verify that $b$ enjoys the skew-symmetry property
\begin{equation}\label{skew}
b(u,v,w)=-b(u,w,v) \quad \forall u,v,w\in H_0^1(\Omega)^d.
\end{equation}
Let us observe that $B(u,v)=\Pi F(u,v)$ for $u\in V,$ $v\in H_0^1(\Omega)^d$.

We shall assume that
$$
\left\|u(t)\right\|_1\le M_1,\quad
\left\|u(t)\right\|_2\le M_2,\qquad 0\le t\le T,
$$
and, for $k\ge 2$ integer,
$$
\sup_{0\le t\le T}
\bigl\| \partial_t^{\lfloor k/2\rfloor} f\bigr\|_{k-1-2{\lfloor k/2\rfloor}}+
\sum_{j=0}^{\lfloor (k-2)/2\rfloor}\sup_{0\le t\le T}
\bigl\| \partial_t^j f\bigr\|_{k-2j-2}
<+\infty,
$$
so that, according to Theorems~2.4 and~2.5 in \cite{heyran0},
there exist positive constants $M_k$ and $K_{k}$ such that for $k\ge 2$
\begin{equation}
\left\| u(t)\right\|_{k} + \left\| u_t(t)\right\|_{k-2}
+\left\| p(t)\right\|_{H^{k-1}/{\mathbb R}} \le M_k \tau(t)^{1-k/2}
\label{eq:u-inf}
\end{equation}
and for $\ k\ge 3$
\begin{equation}
\int_0^t
\sigma_{k-3}(s)
\bigr(\left\| u(s)\right\|_{k}^2 + \left\| u_s(s)\right\|_{k-2}^2
+\left\| p(s)\right\|_{H^{k-1}/{\mathbb R}}^2+
\left\| p_s(s)\right\|_{H^{k-3}/{\mathbb R}} ^2\bigl)\,ds\le
K_{k}^2,
\label{eq:u-int}
\end{equation}
where $\tau(t)=\min(t,1)$ and~$\sigma_n=e^{-\alpha(t-s)}\tau^n(s)$ for some
$\alpha>0$. Observe that, for $t\le T<\infty$, we can take~$\tau(t)=t$ and
$\sigma_n(s)=s^n$. For simplicity, we will take these values of
$\tau$ and~$\sigma_n$. We note that no further
than $k\le 6$ will be needed in the present paper.

Let $\mathcal{T}_{h}=(\tau_i^h,\phi_{i}^{h})_{i \in
I_{h}}$, $h>0$, be a family of partitions of suitable domains $\Omega_h$, where
$h$ is the maximum diameter of the elements $\tau_i^h\in \mathcal{T}_{h}$
and $\phi_i^h$ are the mappings of the reference simplex
$\tau_0$ onto $\tau_i^h$.
We restrict ourselves to quasi-uniform and regular meshes $\mathcal{T}_{h}$.

Let $r \geq 2$, we consider the finite-element spaces
\begin{eqnarray*}
S_{h,r}&=&\big\{ \chi_{h} \in \mathcal{C}(\overline{\Omega}_{h}) \, |
\, {\chi_{h}}{|_{\tau_{i}^{h}}}
\circ \phi^{h}_{i} \, \in \, P^{r-1}(\tau_{0}) \big\},
\nonumber\\
\bet
{S}_{h,r}^0&=&\big\{ \chi_{h} \in \mathcal{C}(\overline{\Omega}_{h}) \, | \,
{\chi_{h}}{|_{\tau_{i}^{h}}} \circ \phi^{h}_{i} \, \in \, P^{r-1}(\tau_{0}) , \, \, \chi_{h}
(x)=0 \,\, \forall\, x \in
\partial\Omega_{h} \big\},
\end{eqnarray*}
where $P^{r-1}(\tau_{0})$ denotes the space of polynomials
of degree at most $r-1$ on $\tau_{0}$. Since we are assuming that
the meshes are quasi-uniform, the
following inverse inequality holds
for each $v_{h} \in ({S}_{h,r}^0)^{d}$
(see, e.g., \cite[Theorem 3.2.6]{ciar0})
\begin{eqnarray}
\label{inv}
\| v_{h} \|_{W^{m,q}(\tau)^{d}} \leq C h^{l-m-d(\frac{1}{q'}-\frac{1}{q})}
\|v_{h}\|_{W^{l,q'}(\tau)^{d}},
\end{eqnarray}
where $0\leq l \leq m \leq 1$, $1\leq q' \leq q \leq \infty$, and $\tau$ is an element in the partition $\mathcal{T}_{h}$.

We shall denote by $(X_{h,r}, Q_{h,r-1})$
the so-called Hood--Taylor element \cite{BF,hood0}, when $r\ge 3$, where
$$
X_{h,r}=\left({S}_{h,r}^0\right)^{d},\quad
Q_{h,r-1}=S_{h,r-1}\cap L^2(\O_{h})/{\mathbb R},\quad r
\ge 3,
$$
and the so-called mini-element~\cite{Brezzi-Fortin91} when $r=2$,
where $Q_{h,1}=S_{h,2}\cap L^2(\O_{h})/{\mathbb R}$, and
$X_{h,2}=({S}_{h,2}^0)^{d}\oplus{\mathbb B}_h$. Here,
${\mathbb B}_h$ is spanned by the bubble functions $b_\tau$,
$\tau\in\mathcal{T}_h$, defined by
$b_\tau(x)=(d+1)^{d+1}\lambda_1(x)\cdots
 \lambda_{d+1}(x)$,  if~$x\in \tau$ and 0 elsewhere,
where  $\lambda_1(x),\ldots,\lambda_{d+1}(x)$ denote the
 barycentric coordinates of~$x$.
For these mixed elements a uniform inf-sup condition is satisfied
(see \cite{BF}); that is,
there exists a constant $\beta>0$ independent of the mesh grid size $h$ such that
\begin{equation}\label{lbbh}
\displaystyle \inf_{q_{h}\in Q_{h,r-1}}\displaystyle\sup_{v_{h}\in X_{h,r}}
\frac{(q_{h},\nabla \cdot v_{h})}{\|v_{h}\|_{1}
\|q_{h}\|_{L^2/{\mathbb R}}} \geq \beta.
\end{equation}
The approximate velocity belongs to the discretely
divergence-free space
$$
V_{h,r}=X_{h,r}\cap \Bigl\{ \chi_{h} \in H^{1}_{0}(\Omega_{h}) \mid
(q_{h}, \nabla\cdot\chi_{h}) =0 \quad\forall q_{h} \in Q_{h,r-1} \Bigr\}.
$$
We observe that, for the
Hood--Taylor element, $V_{h,r}$ is not a subspace of $V$.
Let $\Pi_{h}:L^{2}(\O )^{d}\lor V_{h,r}$ be the discrete Leray's
projection defined by
$$
(\Pi_{h}u,\chi_{h} )=(u,
\chi_{h})\quad \forall\chi_{h}\in V_{h,r}.
$$
We will use the following well known bounds
\begin{equation}
\left\|(I-\Pi_{h})u\right\|_j \le Ch^{l-j}\left\|u\right\|_l,\qquad
1\le l\le r,\quad j=0,1 \label{eq:error-Pi_h}.
\end{equation}
Assuming that $\Omega$ is has a smooth enough boundary, we also have
\begin{equation}
 \bigl\|A^{-m/2}\Pi(I-\Pi_{h})u\bigr\|_0\le Ch^{l+\min(m,r-2)}
\left\|u\right\|_l,\qquad 1\le l\le r,\quad m=1,2.\quad
\label{eq:error-A-1Pi_h}
\end{equation}
Since $(A_{h}^{-1/2} \Pi_{h} f, v_h)=(f,
 A_{h}^{-1/2}v_h)$, for all $v_h\in V_{h,r}$, it follows that
\begin{equation}\label{eq:nch11}
\|A_{h}^{-1/2} \Pi_{h} f\|_{0}\leq C\| f\|_{-1}.
\end{equation}
Moreover it holds for $f\in
L^2(\Omega)^d$, see \cite{jbj_regularity}:
\begin{eqnarray}
\|A_{h}^{-s/2}\Pi_{h}f\|_{0}&\leq& Ch^{s}\|f\|_{0}+\|A^{-s/2}\Pi f\|_{0}
\quad s=1,2.\label{nch1}
\end{eqnarray}
Let~$\mathcal{A}$ denote either $\mathcal{A}=A$ or~$\mathcal{A}=A_h$. Notice that
both are positive self-adjoint operators with compact resolvent in~$H$ and~$V_h$,
respectively.
Let us consider then for $\alpha\in {\mathbb R}$ and $t>0$ the operators
$\mathcal{A}^{\alpha}$ and~$e^{-t\mathcal{A}}$, which are defined
by means of the spectral properties of~${\cal A}$
(see, e.g., \cite[p.~33]{Constantin-Foias}, \cite{Fujita-Kato}).
An easy calculation shows that
\begin{equation}\label{lee1}
\|\mathcal{A}^{\alpha}e^{-t\mathcal{A}}\|_{0}\leq
(\alpha e^{-1})^\alpha
t^{-\alpha}, \qquad \alpha\ge 0,\ t>0,
\end{equation}
where, here and in what follows, $\left\|\cdot\right\|_0$ when applied to an
operator denotes the associated operator norm.
Also, using
the change of
variables $\tau=s/t$, it is easy to show that
\begin{equation}
\int_0^ts^{-1/2}\bigl\| A_h^{1/2}e^{-(t-s)A_h}\bigr\|_0\,ds
\le \frac{1}{\sqrt{2e}} B\left(\frac{1}{2},\frac{1}{2}\right),
\label{labeta}
\end{equation}
where $B$ is the Beta function (see, e.g., \cite{Abramowitz-Stegun}).


\section{Semi-discretization in space. The first two-grid algorithm.}
\label{se:3}
In this section we carry out the error analysis of the first two-grid algorithm for the Hood-Taylor
mixed finite element with $r=3$ or $4$. At the end of the section we include the results
that can be obtained for the mini-element with a similar but simpler analysis than the
one showed along the section.

The first algorithm we consider is the following. Let us choose $h<H$ so that
$V_{H,r}\subset V_{h,r}$. Then, in the first level we compute the
standard mixed finite-element approximation to (\ref{onetwo})--(\ref{ic}). That is,
given $u_H(0)=\Pi_H(u_{0})$, we compute
$u_{H}(t)\in X_{H,r}$ and $p_{H}(t)\in Q_{H,r-1}$, $t\in(0,T]$, satisfying, for all $\phi_{H} \in X_{H,r}$
and $\psi_{H} \in Q_{H,r-1}$
\begin{align}\label{ten}
(\dot u_{H}, \phi_{H})
+ ( \nabla u_{H}, \nabla \phi_{H}) + b(u_{H}, u_{H}, \phi_{H}) +
( \nabla p_{H}, \phi_{H}) & = (f, \phi_{H})\\
\bet
(\nabla \cdot u_{H}, \psi_{H}) & = 0.
\label{ten2}
\end{align}
In the second level we solve a linearized problem on a finer grid and given $\tilde u_h(0)=\Pi_hu_0$ we compute
$\tilde u_{h}(t)\in X_{h,r}$ and $\tilde p_{h}(t)\in Q_{h,r-1}$, $t\in(0,T]$, satisfying, for all $\phi_{h} \in X_{h,r}$
and $\psi_{h} \in Q_{h,r-1}$
\begin{align}\label{ten_p1}
(\dot  {\tilde u}_{h}, \phi_{h})
+ ( \nabla \tilde u_{h}, \nabla \phi_{h}) + (u_{H}\cdot\nabla \tilde u_{h}, \phi_{h}) +
( \nabla \tilde p_{h}, \phi_{h}) & = (f, \phi_{h})\\
\bet
(\nabla \cdot \tilde u_{h}, \psi_{h}) & = 0.
\label{ten2_p1}
\end{align}
To obtain the error bounds for $(\tilde u_h,\tilde p_h)$ we will follow the error analysis of \cite{jbj_regularity}
and introduce an auxiliary approximation (see~\cite[Section 4.1]{jbj_regularity}).
For a $u$ and $p$ solution of (\ref{onetwo})--(\ref{ic}) let us
consider the approximations
$v_h:[0,T]\longrightarrow X_{h,r}$ and $g_h:[0,T]\longrightarrow
Q_{h,r-1}$, respectively, solutions of
\begin{align}\label{tenv}
(\dot v_{h}, \phi_{h}) + ( \nabla v_{h}, \nabla \phi_{h}) +
( \nabla g_{h}, \phi_{h}) & = (f, \phi_{h}) - b(u, u, \phi_{h})\\
\bet
(\nabla \cdot v_{h}, \psi_{h}) & = 0,\label{ten2v}
\end{align}
for all $\phi_{h} \in X_{h,r}$ and $\psi_{h} \in
Q_{h,r-1}$,
with initial condition $v_h(0)=\Pi_h u_0$.
We will also use the following notation:
\begin{equation}\label{z_h}
z_h=\Pi_{h} u- v_h.
\end{equation}
Next, we state some lemmas that are needed in the proof of the main theorems.
The first one summarizes previous results.

\begin{lema}\label{cotas_gal}
Let $(u,p)$ be the solution
of {\rm (\ref{onetwo})--(\ref{ic})}. There exists a positive
constant $C$ such that the discrete velocity~$v_h$
defined by (\ref{tenv})-(\ref{ten2v}) and the Hood--Taylor
element approximation to $u$, $u_{H}$, satisfy the following bounds for
$j=0,1$, and~$t\in(0,T]$:
\begin{eqnarray}\label{supercuad}
 \|v_{H}(t)-u_{H}(t)\|_{j} &\le&
 \frac{C}{t^{(r-2)/2}}|\log(H)|
 H^{r+1-j},\quad  3\le r\le 4,\qquad
\\
\label{gal_cuad}
  \|u_{H}(t)-u(t)\|_{j} &\le&
 \frac{C}{t^{(r-2)/2}}
 H^{r-j},\quad 2\le r\le 5,
 \\
 \label{gal_int}
 \int_0^t \|u_{H}(s)-u(s)\|_{j}^2\,ds &\le& CH^{2(3-j)},\quad 3\le r\le 4.
\end{eqnarray}
\end{lema}
\begin{proof}
The bound  (\ref{supercuad}) is proved
in Theorems 4.7 and 4.15 in \cite{jbj_regularity}.
The case $j=0$ in (\ref{gal_int}) is proved in~\cite[Theorem 3.1]{heyran0}
and~\cite[Theorem 3.1]{heyran2}. The case $j=1$ follows from the case
$j=0$ by applying  (\ref{inv}) and~(\ref{eq:error-Pi_h}), see also Corollaries 4.8 and 4.16 in \cite{jbj_regularity}.
Finally, (\ref{gal_int}) is proved in~Lemmas~5.1 and~5.2 in~\cite{heyran2}.
\end{proof}
For the convenience of the reader, we will reproduce here the following two Lemmas,
the first one from \cite{bbj}
and the second one from \cite{jbj_regularity}.
\begin{lema}
\label{l4b}
For any $f \in C([0,T];L^{2}(\O)^{d})$, the
following estimate holds for all $t\in [0,T]$:
\begin{eqnarray*}
\int_{0}^{t} \big\|A_{h}e^{- (t-s)A_{h}}\Pi_{h}f(s)\big\|_{0}ds
\leq {C}
|\log(h)|\displaystyle\max_{0\leq t\leq T}\|f(t)\|_{0}.
\end{eqnarray*}
\end{lema}
\begin{lema}
\label{le:estz_h}
Let $(u,p)$ be the solution of {\rm (\ref{onetwo})--(\ref{ic})}. Then,
there exists a positive constant $C$ such that the error
$z_h=\Pi_{h}u-v_h$ in (\ref{z_h}) satisfies
the following bound:
\begin{align}
\label{eq:cota-z-1}
\|A_h^{(-1+j)/2}z_h\|_0&\le \frac{C}{t^{(r-2)/2}} h^{r+1-j},&& j=0,1,2,&&r\ge 3.
\end{align}
\end{lema}

\begin{lema}
\label{le:lips01}
For each $\alpha>0$ there exist positive constants
$K>0$ and $h_0$ depending on $\alpha$ and $M_2$ such that, for
$h<H<h_0$,
$h_1=h_2=h$ or $\{h_1,h_2\}=\{h,H\}$,
and
every $w^1_{h_1}(\cdot)\in V_{h_1,r}$ and $,w^2_{h_2}(\cdot)\in V_{h_2,r}$
satisfying the threshold condition
\begin{equation}
\left\| u(t)-w^1_{h_1}(t)\right\|_j\le\alpha h_1^{2-j},\quad \left\|
u(t)-w^2_{h_2}(t)\right\|_j\le \alpha h_2^{2-j},\quad j=0,1,\ t\in[0,T],
\label{eq:threshold}
\end{equation}
for $w_h(t)\in H^1_0(\Omega)$, $t\in [0,T]$, satisfying
$$\left\| w_{h}(t)\right\|_j\le 2\alpha \max(h_1,h_2)^{2-j},$$
the following bounds hold:
\begin{eqnarray}
\left\| {\cal F}(w_h,w^2_{h_2})\right\|_0+\left\|{\cal F}(w^1_{h_1},w_h)\right\|_0
&\le& K \left\|
w_h\right\|_1, \label{eq:lips0b}
\\
\bet
\left\|{\cal F}(w_h,w^2_{h_2})\right\|_{-1}+\left\|{\cal F}(w^1_{h_1},w_h)\right\|_{-1} &\le& K
\left\| w_h\right\|_0, \label{eq:lips1b}\\
\bet
\left\| {\cal B}_h(w_h,w^2_{h_2})\right\|_0+\left\|{\cal B}_h(w^1_{h_1},w_h)\right\|_0 &\le& K \left\|
w_h\right\|_1, \label{eq:lips0}
\\
\bet
\bigl\| A_h^{-1/2}({\cal B}_h(w_h,w^2_{h_2}))\bigr\|_0+\bigl\|A_h^{-1/2}({\cal B}_h(w^1_{h_1},w_{h}))\bigr\|_0 &\le& K
\left\| w_h\right\|_0, \label{eq:lips1}
\end{eqnarray}
where ${\cal F}(u,v)$ can be either $(u\cdot\nabla) v+\frac{1}{2}(\nabla\cdot u)v$ or $(u\cdot\nabla) v$,
and ${\cal B}_h=\Pi_h{\cal F}$.
\end{lema}
\begin{proof}
The proof of
the present lemma can be found in
that of~Lemma~4.4 in \cite{jbj_regularity}
for ${\cal F}(u,v)=(u\cdot\nabla) v+\frac{1}{2}(\nabla\cdot u)v$ and
$w_h\in V_{h,r}$. With obvious
changes, the proof is also valid when ${\cal F}(u,v)=(u\cdot\nabla) v$, as well
as when $w_h\not\in V_{h,r}$.
\end{proof}
\begin{remark}
\label{re:1}
We will apply the above inequalities for $w_h=w_{h_1}^1-w_{h_2}^2$,
$w_h=w_{h_1}^1-u$ and
$w_h=w_{h_2}^2-u$. Let us also remark that the Lemma~\ref{le:lips01} also
holds if either $w_{h_1}^1$ or
$w_{h_2}^2$ is replaced by~$u$.
In what follows we will apply Lemma~\ref{le:lips01} to  $u_h$ and
$v_h$  both satisfying the threshold condition  (\ref{eq:threshold}) for an appropriate
value of~$\alpha$ (see \cite[Remark 4.1]{jbj_regularity}).
\end{remark}
\begin{lema}
\label{le:l3}
For $v \in (H^{2}(\Omega))^d \cap V$ there exists a positive constant $K=K(\|v\|_{2})$
such that $w \in H^{1}_{0}(\Omega)^{d}$ the following bound holds
for
$
e=v-w
$:
\begin{equation}\label{est4}
\phantom{aa}\big\|A^{-1}\Pi[ {\cal F}(v,e) +{\cal F}(e,v )]\big\|_{0}
\leq K
 \| v-w \|_{-1},
\end{equation}
where ${\cal F}(u,v)$ can be either $(u\cdot\nabla) v+\frac{1}{2}(\nabla\cdot u)v$ or $(u\cdot\nabla) v$.
\end{lema}
\begin{proof}
The proof of this result when ${\cal F}(u,v)=(u\cdot\nabla) v+\frac{1}{2}(\nabla\cdot u)v$ can be found as part of the proof of~\cite[Lemma 3.4]{bbj}. With
obvious changes, the proof is also valid when ${\cal F}(u,v)=(u\cdot\nabla) v$.
\end{proof}

Let us observe that the approximation over the finer grid $\tilde u_h$ and the recently defined
$v_h$ satisfy
\begin{gather}
\label{atenop}
 \dot{\tilde u}_{h}+ A_{h} \tilde u_{h}+ \Pi_h(u_H\cdot \nabla \tilde u_h)= \Pi_{h}f,\qquad
 u_h(0)=\Pi_{h}u_0, \\
\bet
\label{atenopv}
 \dot{v}_{h}+ A_{h} v_{h}+ \Pi_h(u\cdot \nabla u)= \Pi_{h}f,\qquad
 v_h(0)=\Pi_{h}u_0,
 \end{gather}
respectively. Then $e_h=v_h-\tilde u_h$ satisfies
\begin{eqnarray}\label{e_h}
\dot e_h+ A_h e_h+\Pi_h(u_H\cdot \nabla e_h)=\Pi_h\rho_{h,H},\qquad e_h(0)=0,
\end{eqnarray}
where
$$
\rho_{h,H}=u_H\cdot \nabla v_h-u\cdot \nabla u.
$$
In the proof of Theorem~\ref{teo1} below we will use the following lemmas.

\begin{lema}
\label{le:aux1}
Let $(u,p)$ be the solution of {\rm (\ref{onetwo})--(\ref{ic})}.
 There exists a positive constant $C$ such that
the  following inequality holds for $r=3,4$:
\begin{eqnarray*}
\|A_h^{-1}\Pi_h\rho_{h,H}\|_0\le
\frac{C}{t^{(r-2)/2}} |\log(H)|H^{r+1},\qquad
t\in (0,T].
\end{eqnarray*}
\end{lema}
\unskip

\begin{proof}
Let us write
$\rho_{h,H}=\rho_{h,H}^1+\rho_{h,H}^2$, where
\begin{equation}\label{decomrho}
\rho_{h,H}^1=((u_H-u)\cdot \nabla v_h),\quad  \rho_{h,H}^2=(u\cdot \nabla (v_h-u)).
\end{equation}
By applying~(\ref{nch1})
we have
$$
\|A_h^{-1}\Pi_h\rho_{h,H}^j\|_0\le C
h^2\|\rho_{h,H}^j\|_0+\|A^{-1}\Pi \rho_{h,H}^j\|_0,\quad j=1,2.
$$
To bound $\|\rho_{h,H}^1\|_0$ let us recall
Remark~\ref{re:1} and~apply~(\ref{eq:lips0b})
to get
\begin{equation}\label{norma0rho1}
\|\rho_{h,H}^1\|_0\le C\|u_H-u\|_1\le C \frac{H^{r-1}}{t^{(r-2)/2}},
\end{equation}
where we have applied (\ref{gal_cuad})  from Lemma~\ref{cotas_gal} in the last inequality.
Applying (\ref{eq:lips0b}) we also get
\begin{equation}\label{norma0rho2}
\|\rho_{h,H}^2\|_0\le C \|v_h-u\|_1\le C \frac{h^{r-1}}{t^{(r-2)/2}},
\end{equation}
where in the last inequality we have applied standard bounds for $\Pi_{h}$ (see
(\ref{eq:error-Pi_h})) together with the estimates  (\ref{eq:cota-z-1}) for $z_h$ in
Lemma~\ref{le:estz_h}.
Let us next bound $\|A^{-1}\Pi \rho_{h,H}^1\|_0$. We will use the decomposition
\begin{equation}\label{decomposition}
\rho_{h,H}^1=((u-u_H)\cdot \nabla)(v_h-u)+((u-u_H)\cdot \nabla)u.
\end{equation}
Then, we obtain
\begin{eqnarray}\label{Amenos1}
\|A^{-1}\Pi \rho_{h,H}^1\|_0&=&\|A^{-1}\Pi \left(((u_H-u)\cdot \nabla) (v_h-u)\right)\|_0\nonumber\\
\qquad
&&{}+\|A^{-1}\Pi \left(((u_H-u)\cdot \nabla) u\right)\|_0.
\end{eqnarray}
To bound the second term in (\ref{Amenos1}) we apply (\ref{est4}) from Lemma~\ref{le:l3} to get
\begin{eqnarray*}
 \|A^{-1}\Pi \left(((u_H-u)\cdot \nabla) u\right)\|_0\le C \|u_H-u\|_{-1}.
\end{eqnarray*}
Applying then (\ref{supercuad}) together with (\ref{eq:error-A-1Pi_h}) and (\ref{eq:cota-z-1})  we get
\begin{eqnarray}\label{menos1u_H}
\|u_H-u\|_{-1}&\le& \|u_H-v_H\|_0+\|v_H-u\|_{-1}\nonumber\\
&\le& \frac{C}{t^{(r-2)/2}}|\log(H)|H^{r+1}+\frac{C}{t^{(r-2)/2}}H^{r+1}.
\end{eqnarray}
To bound the first term in (\ref{Amenos1}) we argue by duality, using (\ref{sob1}), we get
\begin{eqnarray*}
&&\|A^{-1}\Pi \left(((u_H-u)\cdot \nabla) (v_h-u)\right)\|_0\nonumber\\
&&\quad=\sup_{\|\phi\|_0=1}
\left(((u_H-u)\cdot \nabla) (v_h-u),A^{-1}\Pi\phi\right)\nonumber\\
&&\qquad\le\sup_{\|\phi\|_0=1}\|u_H-u\|_0\|v_h-u\|_1\|A^{-1}\Pi\phi\|_\infty\nonumber\\
&& \qquad\le\sup_{\|\phi\|_0=1} C\|u_H-u\|_0\|v_h-u\|_1\|A^{-1}\Pi\phi\|_2\le C\|u_H-u\|_0\|v_h-u\|_1.
\end{eqnarray*}
Now, in view of the case $r=2$ in~(\ref{gal_cuad})
and using again (\ref{eq:error-Pi_h}) and (\ref{eq:cota-z-1})
 we conclude
$$
\|A^{-1}\Pi \left(((u_H-u)\cdot \nabla) (v_h-u)\right)\|_0\le \frac{C}{t^{(r-2)/2}} H^2 h^{r-1}.
$$
Finally, to bound $\|A^{-1}\Pi\rho_{h,H}^2\|_0$ we apply again (\ref{est4}) to bound this norm in terms of $\|v_h-u\|_{-1}$
which, as we shown in (\ref{menos1u_H}), is bounded by $Ct^{(2-r)/2} h^{r+1}$.
\end{proof}
\begin{lema}
\label{le:aux1b}
Let $(u,p)$ be the solution of {\rm (\ref{onetwo})--(\ref{ic})}.
 There exists a positive constant $C$ such that
the  following inequalities hold for $r=3,4$:
\begin{eqnarray}
\|\rho_{h,H}\|_{-1}&\le&
\frac{C}{t^{(r-1)/2}} |\log(H)|H^{r+1}+\frac{C}{t^{(r-2)/2}}h^r, \qquad t\in (0,T],
\label{primer}
\\
\|\rho_{h,H}\|_{-1}&\le&\frac{C}{t^{1/2}}H^3, \qquad t\in (0,T].\label{segun}
\end{eqnarray}
\end{lema}
\begin{proof}
The proof is very similar to the one of the previous lemma. We will prove (\ref{primer}) since
the proof of (\ref{segun}) is completely analogous and yet easier.
We use the decomposition (\ref{decomrho}).

For $\rho_{h,H}^2$ we apply (\ref{eq:lips1b}) to get
$$
\|\rho_{h,H}^2\|_{-1} \le C\|v_h-u\|_0\le C \frac{h^r}{t^{(r-2)/2}},
$$
where we have applied
(\ref{eq:error-Pi_h}) and  (\ref{eq:cota-z-1}) in the last inequality.
For  $\rho_{h,H}^1$ we will use the decomposition (\ref{decomposition}).
For the first term in (\ref{decomposition}) using~(\ref{sob1}) we have
\begin{eqnarray*}
\| ((u_H-u)\cdot \nabla) (v_h-u)\|_{-1}
&=&\sup_{\|\phi\|_1=1}
\left(((u_H-u)\cdot \nabla) (v_h-u),\phi\right)\nonumber\\
&\le&\sup_{\|\phi\|_1=1} \|u_H-u\|_{L^{2d}}\|v_h-u\|_1\|\phi\|_{L^{2d/(d-1)}}\nonumber\\
&\le& C\|u_H-u\|_1\|v_h-u\|_1\le C \frac{H^{r-1}}{t^{(r-2)/2}}\frac{h^2}{t^{1/2}},
\end{eqnarray*}
where we have applied (\ref{gal_cuad}) and (\ref{eq:error-Pi_h}) and  (\ref{eq:cota-z-1}) in the
last inequality. Finally, for the second term using  (\ref{sob1}), (\ref{eq:u-inf}) and (\ref{menos1u_H})
we obtain
\begin{align*}
\|((u_H-u)\cdot \nabla) u\|_{-1}
&
=\sup_{\|\phi\|_1=1}
\left(((u_H-u)\cdot \nabla) u,\phi\right)\nonumber\\
&
\le\sup_{\|\phi\|_1=1} \|u_H-u\|_{-1}\|\phi\nabla u \|_{1}\nonumber\\
&\le \sup_{\|\phi\|_1=1}\|u_H-u\|_{-1}\Bigl(\|\nabla u\|_{W^{1,2d/(d-1)}}\|\phi\|_{L^{2d}}
\\
&\hphantom{\quad\le \sup_{\|\phi\|_0=1}\|u_H-u\|_{-1}\Bigl(\|\Delta u\|_{L^{2d/(d-1)}}
}
{}+\|\nabla u\|_\infty\| \phi\|_1\Bigr)\nonumber\\
&\le \frac{C}{t^{(r-2)/2}}|\log(H)|H^{r+1}\|u\|_3\le \frac{C}{t^{(r-1)/2}}|\log(H)|H^{r+1}.
\end{align*}
\end{proof}
\begin{lema}
\label{le:cota_pre_e1}
Let $(u,p)$ be the solution of {\rm (\ref{onetwo})--(\ref{ic})}. Then there
exists a positive constant $C$ such that
the discrete velocity $v_h$ defined by~$(\ref{atenopv})$
and the approximation to $u$ over the finer grid, $\tilde u_{h}$ satisfy
the following bound:
\begin{eqnarray*}
\|A_h^{l/2}(v_h(t)-\tilde u_h(t))\|_0 \le C H^{3-l},\quad r\ge 3,\quad l=0,1,\quad t\in (0,T].
\end{eqnarray*}
\end{lema}
\unskip

\begin{proof}
Let us consider $y_h(t)=A_h^{l/2}e_h(t)$. From (\ref{e_h}) it follows
that
$$
y_h(t)=\int_0^t e^{-(t-s)A_h}
A_h^{l/2}\Pi_h\bigl((u_H\cdot\nabla) e_h\bigr)\,ds
+ \int_0^t e^{-(t-s)A_h} A_h^{l/2}\Pi_h\rho_{h,H}(s)\,ds.
$$
Applying (\ref{lee1}), and taking into account that as a consequence of~(\ref{eq:lips0b}) and (\ref{eq:lips1}) we have $\|A_h^{(-1+l)/2}\Pi_h((u_H\cdot\nabla) e_h)\|_0\le C\|A_h^{l/2}e_h\|_0$, it follows that
\begin{eqnarray*}
\left\| y_h(t)\right\|_0 \le \int_0^t \frac{C}{\sqrt{t-s}} \|y_h\|_0\,ds  +\biggl\| \int_0^t e^{-(t-s)A_h}
A_h^{l/2}\Pi_h\rho_{h,H}(s)\,ds\biggr\|,
\end{eqnarray*}
so that a generalized Gronwall lemma \cite[pp.~188--189]{Henry} allow us to
write
\begin{equation}
\label{delema8}
\max_{0\le t\le T}\left\| y_h(t)\right\|_0 \le C\max_{0\le t\le T}
\biggl\|\int_0^t e^{-(t-s)A_h} A_h^{l/2}\Pi_h\rho_{h,H}(s)\,ds \biggr\|_0.
\end{equation}
Using (\ref{labeta}) we obtain
\begin{equation}
\label{jul5}
\max_{0\le t\le T}\left\| y_h(t)\right\|_0 \le C
B\left(\frac{1}{2},\frac{1}{2}\right)\max_{0\le s\le T} s^{1/2}\bigl\|A_h^{(-1+l)/2}\Pi_h \rho_{h,H} \bigr\|_0.
\end{equation}
To conclude we apply~(\ref{eq:nch11}) and~(\ref{segun}) from Lemma~\ref{le:aux1b} in the
case $l=0$, and (\ref{norma0rho1}) and (\ref{norma0rho2}) in the case $l=1$.
\end{proof}
\begin{lema}
\label{le:aux2}
Let $(u,p)$ be the solution of {\rm (\ref{onetwo})--(\ref{ic})}. Then, there
exists a positive constant $C$ such that
the discrete velocity $v_h$ defined by~$(\ref{atenopv})$
and the approximation to $u$ over the finer grid, $\tilde u_{h}$ satisfy
the following bound:
\begin{eqnarray*}
\|A_h^{-1/2}(v_h(t)-\tilde u_h(t))\|_0 \le C |\log(H)|H^4,\quad r\ge 3,\quad t\in (0,T].
\end{eqnarray*}
\end{lema}
\unskip

\begin{proof}
The proof follows the steps of the proof of Lemma~4.6 in \cite{jbj_regularity}.
Let us consider $y_h(t)=A_h^{-1/2}e_h(t)$. From (\ref{e_h}) it follows
that
$$
y_h(t)=\int_0^t e^{-(t-s)A_h}
A_h^{-1/2}\Pi_h\bigl((u_H\cdot\nabla) e_h\bigr)\,ds
+ \int_0^t e^{-(t-s)A_h} A_h^{-1/2}\Pi_h\rho_{h,H}(s)\,ds.
$$
Applying (\ref{lee1})  we have that
\begin{eqnarray}\label{eq:antes_gro}
\left\| y_h(t)\right\|_0 &\le& \int_0^t \frac{C}{\sqrt{t-s}} \left\|
A_h^{-1}\Pi_h ((u_H\cdot \nabla)e_h)\right\|_0\,ds  \nonumber\\
&&\quad+\biggl\| \int_0^t e^{-(t-s)A_h}
A_h^{-1/2}\Pi_h\rho_{h,H}(s)\,ds \biggr\|_0.
\end{eqnarray}
Let us now bound $\left\|
A_h^{-1}\Pi_h ((u_H\cdot \nabla)e_h)\right\|_0$. We will argue as in the proof of \cite[(4.23) Lemma 4.4]{jbj_regularity}.
Let us first observe that $h^2\|e_h\|_1\le C\|A_h^{-1/2}e_h\|_0$. Using (\ref{nch1}) we get
\begin{eqnarray*}
\left\|
A_h^{-1}\Pi_h ((u_H\cdot \nabla)e_h)\right\|_0&\le& C h^2\|(u_H\cdot \nabla)e_h\|_0+\|A^{-1}\Pi((u_H\cdot \nabla)e_h)\|_0
\nonumber\\
&\le& C h^2\|e_h\|_1+\|A^{-1}\Pi((u_H\cdot \nabla)e_h)\|_0
\nonumber\\&\le&C \|A_h^{-1/2}e_h\|_0+\|A^{-1}\Pi((u_H\cdot \nabla)e_h)\|_0.
\end{eqnarray*}
Let us now bound the second term on the right hand side above. We write
\begin{eqnarray*}
\|A^{-1}\Pi((u_H\cdot \nabla)e_h)\|_0\le \|A^{-1}\Pi(((u_H-u)\cdot \nabla)e_h)\|_0
+\|A^{-1}\Pi((u\cdot \nabla)e_h)\|_0.
\end{eqnarray*}
For the first term arguing by duality we get
$$
\|A^{-1}\Pi(((u_H-u)\cdot \nabla)e_h)\|_0\le C \|u_H-u\|_0\|e_h\|_1\le C H^2\|e_h\|_1.
$$
For the second one, arguing again by duality and integrating by parts, we get
$$
\|A^{-1}\Pi((u\cdot \nabla)e_h)\|_0\le C\|e_h\|_{-1}\|u\|_2\le C \|y_h\|_0.
$$
We finally obtain
\begin{equation}
\label{jul11}
\left\|
A_h^{-1}\Pi_h ((u_H\cdot \nabla)e_h)\right\|_0\le C \|y_h\|_0+C H^2\|e_h\|_1.
\end{equation}
Going back to (\ref{eq:antes_gro}) we obtain
\begin{eqnarray}
\left\| y_h(t)\right\|_0 &\le& \int_0^t \frac{C}{\sqrt{t-s}} \|y_h\|_0\,ds  +\biggl\| \int_0^t e^{-(t-s)A_h}
A_h^{-1/2}\Pi_h\rho_{h,H}(s)\,ds \biggr\|_0\nonumber\\
&&\quad+C t^{1/2}H^2\max_{0\le s\le t}\|e_h(s)\|_1,
\label{dellema9}
\end{eqnarray}
so that a generalized Gronwall lemma \cite[pp.~188--189]{Henry} allow us to
write
$$
\max_{0\le t\le T}\left\| y_h(t)\right\|_0 \le C\biggl(\max_{0\le t\le T}
\biggl\|\int_0^t \!\!e^{-(t-s)A_h} A_h^{-1/2}\Pi_h\rho_{h,H}\,ds \biggr\|_0\!\!+
H^2\!\!\!\max_{0\le t\le T}\|e_h\|_1\biggr).
$$
Using (\ref{labeta}) we obtain
\begin{equation}
\label{jul12}
\max_{0\le t\le T}\left\| y_h(t)\right\|_0 \le C\biggl(
B\left(\frac{1}{2},\frac{1}{2}\right)\max_{0\le s\le T} s^{1/2}\bigl\|
A_{h}^{-1}\Pi_h \rho_{h,H} \bigr\|_0+H^2\max_{0\le t\le T}\|e_h\|_1\biggr),
\end{equation}
where, by applying Lemmas~\ref{le:aux1} and~\ref{le:cota_pre_e1}  the proof is finished.
\qquad\end{proof}
The proof of the following theorem follows the steps of the proof of \cite[Theorem 4.7]{jbj_regularity}.
\begin{Theorem}
\label{teo1}
Let $(u,p)$ be the solution
of {\rm (\ref{onetwo})--(\ref{ic})}. There exists a positive
constant $C$ such that the discrete velocity~$v_h$
defined by $(\ref{atenopv})$ and the approximation to $u$ over the finer grid, $\tilde u_{h}$, satisfy the following bound:
\begin{equation}\label{supercuad_tg}
 \qquad \|v_{h}(t)-\tilde u_{h}(t)\|_{0} \le
 \frac{C}{t^{1/2}}|\log(h)|\left(|\log(H)|
 H^{4}\right),\quad t\in(0,T],\quad r\ge 3.
\end{equation}
\end{Theorem}
\unskip
\begin{proof}
Let us consider $y_h(t)=t^{1/2}e_h(t)$. From (\ref{e_h}) and an easy calculation we get
$$
\dot y_h+A_h y_h +t^{1/2}\Pi_h(u_H\cdot \nabla e_h)=t^{1/2}\Pi_h \rho_{h,H}+\frac{1}{2 t^{1/2}}e_h.
$$
Then,
$$
\displaylines{
y_h(t)=\int_0^t e^{-A_h(t-s)}s^{1/2}\Pi_h(u_H\cdot \nabla e_h)~ds
\hfill\cr
\bet
\hfill {}+
\int_0^t e^{-A_h(t-s)}\left(s^{1/2}\Pi_h \rho_{h,H}+\frac{1}{2s^{1/2}}e_h\right)~ds.
}
$$
Applying~(\ref{eq:lips1}) we get
\begin{eqnarray}\label{eq:comolip}
\|A_h^{-1/2}\Pi_h (u_H\cdot \nabla )e_h)\|_0\le C\|e_h\|_0.
\end{eqnarray}
Then, using (\ref{lee1}) we obtain
$$
\biggl\|\int_0^t e^{-A_h(t-s)}s^{1/2}\Pi_h (u_H\cdot \nabla e_h)\,ds
\biggr\|_0
\le
C\int_0^t\frac{\|y_h\|_0} {\sqrt{t-s}}~ds.
$$
Applying a generalized
Gronwall lemma \cite[pp.~188--189]{Henry}, it follows that
\begin{align}
 \max_{0\le t\le T}\left\| y_h(t)\right\|_0
 \le &C \biggl(\max_{0\le
t\le T} \biggl\| \int_0^t\! e^{-A_h(t-s)}s^{1/2}\Pi_h\rho_{h,H}\,ds \biggr\|_0
\nonumber\\
{}&+ \max_{0\le s\le t}\biggl\|\int_0^t\!
e^{-A_h(t-s)}\frac{e_h}{s^{1/2}}\,ds \biggr\|_0\biggr).
\label{delteo1}
\end{align}
Applying now Lemma~\ref{l4b} and (\ref{labeta}) we have
\begin{align*}
 \max_{0\le t\le T}\left\| y_h(t)\right\|_0 \le &
 C\Bigl( |\log(h)|
\max_{0\le t\le T} \bigl\|s^{1/2}A_h^{-1}\Pi_h\rho_{h,H}(s)\bigr\|_0 \\
&{}+
B\left(\frac{1}{2},\frac{1}{2}\right) \max_{0\le t\le
T}\bigl\|A_h^{-1/2}e_h(s)\bigr\|_0\Bigr),
\end{align*}
where Lemmas~\ref{le:aux1} and~\ref{le:aux2} finish the proof.
\end{proof}
\begin{Theorem}
\label{teo2}
Let $(u,p)$ be the solution
of {\rm (\ref{onetwo})--(\ref{ic})}. There exists a positive
constant $C$ such that the discrete velocity~$v_h$
defined by $(\ref{atenopv})$ and the approximation to $u$ over the finer grid, $\tilde u_{h}$, satisfy the following bound for $r\ge 3$
\begin{equation}\label{supercuad_tg_1}
 \qquad \|v_{h}(t)-\tilde u_{h}(t)\|_{1} \le
 \frac{C}{t}|\log(h)|\left(|\log(H)|
 H^{4}+T^{1/2}h^3\right),\quad t\in(0,T].
\end{equation}
\end{Theorem}
\unskip
\begin{proof}
Let us define $y_h(t)=t A_h^{1/2} e_h(t)$, where $e_h(t)=v_h(t)-\tilde u_h(t)$.
Arguing exactly as in the proof
of Theorem~\ref{teo1}, instead of~(\ref{delteo1}) we now arrive at
\begin{align*}
 \max_{0\le t\le T}\left\| y_h(t)\right\|_0
\le &C \biggl( \max_{0\le t\le T}
 \biggl\| \int_0^t\! e^{-A_h(t-s)}s A_h^{1/2}\Pi_h\rho_{h,H}\,ds \biggr\|_0
\\
&{}+  \max_{0\le t\le T}\biggl\|\int_0^t\!
e^{-A_h(t-s)}A_h^{1/2}{e_h}\,ds \biggr\|_0\biggr).
\end{align*}
Applying now Lemma~\ref{l4b} we  get
$$
\max_{0\le t\le T}
\|y_h(t)\|_0\le C|\log(h)|\left(\max_{0\le t\le T} \left\|s A_h^{-1/2}\Pi_h \rho_{h,H}\right\|_0
+\max_{0\le t \le T}\left\|A_h^{-1/2} e_h\right\|_0\right),
$$
where Lemmas~\ref{le:aux1b} and~\ref{le:aux2} finish the proof.
\end{proof}

\begin{lema}
\label{le:Cons-4.2}
Let $(u,p)$ be the solution of {\rm (\ref{onetwo})--(\ref{ic})}. Then there
exists a positive constant $C$ such that
the discrete velocity $v_h$ defined by~$(\ref{atenopv})$
and the approximation to $u$ over the finer grid, $\tilde u_{h}$ satisfy
the following bound:
\begin{eqnarray*}
\|A_h^{-1}(v_h(t)-\tilde u_h(t))\|_0 \le C H^5,\quad r\ge 4,\quad t\in (0,T].
\end{eqnarray*}
\end{lema}
\begin{proof}
Let us consider $y_h(t)=A_h^{-1}e_h(t)$. From (\ref{e_h}) it follows
that
$$
y_h(t)=\int_0^t e^{-(t-s)A_h}
A_h^{-1}\Pi_h\bigl((u_H\cdot\nabla) e_h\bigr)\,ds
+ \int_0^t e^{-(t-s)A_h} A_h^{-1}\Pi_h\rho_{h,H}(s)\,ds.
$$
We first observe that arguing exactly as in the proof of \cite[Lemma 4.13]{jbj_regularity}
we get
\begin{align}
\left\| e^{-(t-s)A_h}
A_h^{-1}\Pi_h\bigl((u_H\cdot\nabla) e_h\bigr)\right\|_0\le &C \left(\frac{1}{\sqrt{t-s}}+\frac{1}{\sqrt{s}}\right)\|A_h^{-1}e_h\|_0\nonumber\\
&{}+
C \frac{H^3}{\sqrt{s}} \|e_h\|_1.
\label{jul6}
\end{align}
Then,
\begin{eqnarray*}
\|y_h(t)\|_0&\le& C\int_0^t \left(\frac{1}{\sqrt{t-s}}+\frac{1}{\sqrt{s}}\right)\|y_h(s)\|_0~ds\nonumber\\
&&+\left\|\int_0^te^{-(t-s)A_h} A_h^{-1}\Pi_h\rho_{h,H}(s)\,ds\right\|+C t^{1/2}H^3\max_{0\le s\le t}\|e_h(s)\|_1
\end{eqnarray*}
Applying now \cite[Lemma 4.9]{jbj_regularity} we get
\begin{align}
\max_{0\le t\le T}\|y_h(t)\|_0\le &C\max_{0\le t\le T}\left\|\int_0^te^{-(t-s)A_h} A_h^{-1}\Pi_h\rho_{h,H}(s)\,ds\right\|
\nonumber\\
&{}+C H^3 \max_{0\le t\le T}\|e_h(t)\|_1.
\label{jul8}
\end{align}
For the second term on the right-hand-side above we apply Lemma~\ref{le:cota_pre_e1}.
For the first one we use the decomposition
$$
\rho_{h,H}=\rho_{h,H}^1+\rho_{h,H}^2,\quad \rho_{h,H}^1=((u_H-u)\cdot \nabla (v_h-u))+((u_H-u)\cdot \nabla u),
$$
We now argue exactly as in \cite[(4.60) in Lemma 4.14]{jbj_regularity},
replacing one of the occurrences of~$z$ there by~$u-u_H$ and
making use of~(\ref{gal_cuad}) and~(\ref{gal_int}) with $h$ replaced
by~$H$. This will
allow us  to obtain
$$
\max_{0\le t\le T}\left\|\int_0^te^{-(t-s)A_h} A_h^{-1}\Pi_h\rho_{h,H}(s)\,ds\right\|\le C H^5,
$$
which concludes the proof.
\end{proof}
\begin{Theorem}
\label{teo3}
Let $(u,p)$ be the solution
of {\rm (\ref{onetwo})--(\ref{ic})}. There exists a positive
constant $C$ such that the discrete velocity~$v_h$
defined by $(\ref{atenopv})$ and the approximation to $u$ over the finer grid, $\tilde u_{h}$, satisfy the following bound:
\begin{equation}\label{supercuad_tg}
 \qquad \|v_{h}(t)-\tilde u_{h}(t)\|_{0} \le
 \frac{C}{t}|\log(h)|\left(|\log(H)|
 H^{5}\right),\quad t\in(0,T],\quad r\ge 4.
\end{equation}
\end{Theorem}
\unskip
\begin{proof}
Let us define $y_h(t)=t e_h(t)$, where $e_h(t)=v_h(t)-\tilde u_h(t)$.
Arguing as in the proof
of Theorem~\ref{teo1}, instead of~(\ref{delteo1}) we now arrive at
\begin{align*}
 \max_{0\le t\le T}\left\| y_h(t)\right\|_0  \le &C \biggl(\max_{0\le
t\le T} \biggl\| \int_0^t\! e^{-A_h(t-s)}s\Pi_h\rho_{h,H}\,ds \biggr\|_0
\nonumber
\\&{}+ \max_{0\le t\le T}\biggl\|\int_0^t\|
e^{-A_h(t-s)} e_h\,ds \biggr\|_0\biggr).
\end{align*}
As in the proof of Theorem~\ref{teo2}, applying now Lemma~\ref{l4b}
to both terms on the right-hand side above we
get
$$
\max_{0\le t\le T} \|y_h(t)\|_0\le C|\log(h)|\left(\max_{0\le t\le T} \left\|s A_h^{-1}\Pi_h \rho_{h,H}\right\|_0
+\max_{0\le t \le T}\left\|A_h^{-1} e_h\right\|_0\right).
$$
where now Lemmas~\ref{le:aux1} and~\ref{le:Cons-4.2} finish the proof.
\end{proof}
\begin{lema}
\label{le:ultipri}
Let $(u,p)$ be the solution of {\rm (\ref{onetwo})--(\ref{ic})}. Then there
exists a positive constant $C$ such that
the discrete velocity $v_h$ defined by~$(\ref{atenopv})$
and the approximation to $u$ over the finer grid, $\tilde u_{h}$ satisfy
the following bound:
\begin{eqnarray*}
\|A_h^{-1/2}(v_h(t)-\tilde u_h(t))\|_0 \le \frac{C}{t^{1/2}} |\log(h)|H^5,\quad r\ge 4,\quad t\in (0,T].
\end{eqnarray*}
\end{lema}
\unskip

\begin{proof}
Setting $y_h(t)= t^{1/2}A_h^{-1/2} e_h(t)$ and
arguing exactly as in the proof of Lemma~\ref{le:aux2}, instead
of~(\ref{dellema9}) we now obtain
\begin{align*}
\left\| y_h(t)\right\|_0 \le& \int_0^t \frac{C}{\sqrt{t-s}} \|y_h\|_0\,ds  +\biggl\| \int_0^t e^{-(t-s)A_h}
s^{1/2}A_h^{-1/2}\Pi_h\rho_{h,H}(s)\,ds \biggr\|_0\nonumber\\
&{}+\biggl\|\int_0^t e^{-(t-s)A_h} A_h^{-1/2}\frac{e_h(s)}{2 s^{1/2}}~ds\biggr\|_0 +C H^2\max_{0\le s\le t}s^{1/2}\|e_h(s)\|_1,
\end{align*}
so that a generalized Gronwall lemma \cite[pp.~188--189]{Henry} allow us to
write
\begin{align*}
\max_{0\le t\le T}\left\| y_h(t)\right\|_0 \le &C\max_{0\le t\le T}
\biggl\|\int_0^t e^{-(t-s)A_h} s^{1/2}A_h^{-1/2}\Pi_h\rho_{h,H}(s)\,ds \biggr\|_0\nonumber\\
{}+\max_{0\le t\le T}&\biggl\|\int_0^t e^{-(t-s)A_h} A_h^{-1/2}\frac{e_h(s)}{2 s^{1/2}}~ds\biggr\|_0+
C H^2\!\!\max_{0\le t\le T}t^{1/2}\|e_h(t)\|_1.
\end{align*}
Using (\ref{labeta}) we obtain
\begin{eqnarray*}
\max_{0\le t\le T}\left\| y_h(t)\right\|_0&\le& C B\left(\frac{1}{2},\frac{1}{2}\right)
\left(\max_{0\le s\le T}\|s A_h^{-1}\Pi_h \rho_{h,H}\|_0+\max_{0\le s\le T}\|A_h^{-1}e_h\|_0\right)\nonumber\\
&&+
C H^2\max_{0\le t\le T}t^{1/2}\|e_h\|_1.
\end{eqnarray*}
For the first two terms on the right-hand-side above we apply Lemmas~\ref{le:aux1} and~\ref{le:Cons-4.2} respectively.
For the last term we observe that denoting by $y_h(t)=t^{1/2}A_h^{1/2}e_h$ and
arguing as in Theorem~\ref{teo1} we get
\begin{align*}
\max_{0\le t\le T}\left\| y_h(t)\right\|_0\le &C\Bigl( |\log(h)|
\max_{0\le t \le T}\|t^{1/2}A_h^{-1/2}\Pi_h \rho_{h,H}\|_0
\\
&{}+ B\left(\frac{1}{2},\frac{1}{2}\right)\max_{0\le t\le T}\|e_h\|_0\Bigr),
\end{align*}
so that applying now~(\ref{eq:nch11}) and~(\ref{segun}) to bound the first term on
the right-hand side above, and the case $l=0$ in Lemma~\ref{le:cota_pre_e1} for
the second one, the proof is completed.
\end{proof}
\begin{Theorem}
\label{teo4}
Let $(u,p)$ be the solution
of {\rm (\ref{onetwo})--(\ref{ic})}. There exists a positive
constant $C$ such that the discrete velocity~$v_h$
defined by $(\ref{atenopv})$ and the approximation to $u$ over the finer grid,
$\tilde u_{h}$, satisfy the following bound for $r\ge 4$:
\begin{equation}\label{supercub_tg_1}
 \qquad \|v_{h}(t)-\tilde u_{h}(t)\|_{1} \le
 \frac{C}{t^{3/2}}|\log(h)|\left(|\log(h)|
 H^{5}+T^{1/2}h^4\right),\quad t\in(0,T].
\end{equation}
\end{Theorem}
\unskip
\begin{proof}
Let $y_h(t)=t^{3/2}A_h^{1/2}e_h$ and argue
as in the proof of Theorem~\ref{teo2} to get
$$
\|y_h(t)\|_0\le C|\log(h)|\left(\max_{0\le t\le T} \left\|t^{3/2} A_h^{-1/2}\Pi_h \rho_{h,H}\right\|_0
+\max_{0\le t \le T}\left\|t^{1/2}A_h^{-1/2} e_h\right\|_0\right),
$$
for $t\in(0,T]$.
To bound the first term on the right-hand side above we apply~(\ref{eq:nch11}) and~(\ref{primer}), and for the second one we apply Lemma~\ref{le:ultipri}.
\end{proof}
We now summarize the main results of the section in the following theorem.
\begin{Theorem}
\label{teotodo}
Let $(u,p)$ be the solution
of {\rm (\ref{onetwo})--(\ref{ic})}. There exists a positive
constant $C$ such that the approximation to $u$ over the finer grid, $\tilde u_{h}$, satisfy the following bounds for $r= 3,4$ and $t\in(0,T]$:
\begin{align*}
\|u(t)-\tilde u_h(t)\|_0\le& \frac{C}{t^{(r-2)/2}}|\log(h)||\log(H)|H^{r+1}+\frac{C}{t^{(r-2)/2}}h^r.
\\
\|u(t)-\tilde u_h(t)\|_1\le&\frac{C}{t^{(r-1)/2}}|\log(h)|\left(|\log(h)|H^{r+1}+T^{1/2}h^r\right)+\frac{C}{t^{(r-2)/2}}h^{r-1},
\end{align*}
where in the last inequality we can replace the second $|\log(h)|$ by $|\log(H)|$ in the case $r=3$.
\end{Theorem}
\unskip
\begin{proof}
We use the decomposition $u-\tilde u_h=(u-v_h)+(v_h-\tilde u_h)$. To bound the first term we apply (\ref{eq:error-Pi_h}) and (\ref{eq:cota-z-1})
while for the second we apply Theorems~\ref{teo1}-\ref{teo4}.
\end{proof}

Now, we get the error bounds for the pressure. We begin with some error estimates for the time derivative of $v_h-\tilde u_h$.
\begin{lema}
\label{le:devt}
Let $(u,p)$ be the solution of {\rm (\ref{onetwo})--(\ref{ic})}. Then there
exists a positive constant $C$ such that
the discrete velocity $v_h$ defined by~$(\ref{atenopv})$
and the approximation to $u$ over the finer grid, $\tilde u_{h}$ satisfy
the following bound for $r=3,4$:
\begin{eqnarray}\label{cotae_t}
\|\dot v_h(t)-\dot {\tilde u}_h(t)\|_{-1}\le \frac{C}{t^{(r-1)/2}}|\log(h)|\left(|\log(h)|H^{r+1}+h^r\right),\ t\in(0,T].
\end{eqnarray}
In the case $r=3$ the second $\log(h)$ can be replaced by $\log(H)$.
\end{lema}
\begin{proof}
Using (\ref{e_h}) and taking into account that $\|\dot e_h\|_{-1}\le C\|A_h^{-1/2}\dot e_h\|_0$ we obtain
\begin{eqnarray*}
\|\dot e_h\|_{-1}&\le& \|A_h^{1/2}e_h\|_0+\|A_h^{-1/2}\Pi_h\left((u_H\cdot\nabla)e_h+\rho_{h,H}\right)\|_0
\nonumber\\
&\le&\|e_h\|_1+C \|e_h\|_0+\|A_h^{-1/2}\Pi_h\rho_{h,H}\|_0,
\end{eqnarray*}
after using (\ref{eq:comolip}) in the last inequality.
Applying now Theorems~\ref{teo1}-\ref{teo4} together with~(\ref{eq:nch11})
and~(\ref{primer}) we reach (\ref{cotae_t}).
\end{proof}

The following Lemma is proved in \cite[Corollary 4.19]{jbj_regularity} and
Proposition~3.1 in~\cite{heyran0} and \cite{heyran2} (see also \cite{bjj} and
\cite{jbj_fully}).
\begin{lema}\label{lag_h}
Let $(u,p)$ be the solution of {\rm (\ref{onetwo})--(\ref{ic})} and let  $(u_h,p_h)$
and~$(v_h,g_h)$
the approximations defined
in (\ref{ten})-(\ref{ten2}) and (\ref{tenv})--(\ref{ten2v}), respectively.
Then, the following bound holds for $r=2,3,4$
\begin{align}\label{eq_lap_h}
\|p_h(t)-p(t)\|_{L^2/{\Bbb R}}\le \frac{C}{t^{(r-2)/2}}h^{r-1},\quad t\in(0,T],
\\
\label{eq_lag_h}
\|g_h(t)-p(t)\|_{L^2/{\Bbb R}}\le \frac{C}{t^{(r-2)/2}}h^{r-1},\quad t\in(0,T].
\end{align}
\end{lema}
\begin{Theorem}
\label{teo5}
Let $(u,p)$ be the solution
of {\rm (\ref{onetwo})--(\ref{ic})}. There exists a positive
constant $C$ such that the approximation to $p$ over the finer grid, $\tilde p_{h}$, satisfies the following bound
for $t\in(0,T]$ and $r=3,4$:
\begin{eqnarray*}
\|\tilde p_h(t)-p(t)\|_{L^2/{\Bbb R}}\le \frac{C}{t^{(r-2)/2}}h^{r-1}+\frac{C}{t^{(r-1)/2}}|\log(h)|\left(|\log(h)|H^{r+1}+h^r\right).
\end{eqnarray*}
In the case $r=3$ the second $\log(h)$ can be replaced by $\log(H)$.
\end{Theorem}
\begin{proof}
We use the decomposition
$$
\|\tilde p_h-p\|_{L^2/{\Bbb R}} \le \|\tilde p_h-g_h\|_{L^2/{\Bbb R}}+\| g_h-p\|_{L^2/{\Bbb R}}.
$$
To bound the second term on the right-hand-side above we apply (\ref{eq_lag_h}). For the first one subtracting
(\ref{tenv}) from (\ref{ten_p1}) and applying the inf-sup condition (\ref{lbbh}) we obtain
\begin{eqnarray*}
\beta \|\tilde p_h-g_h\|_{L^2/{\Bbb R}}\le \|\tilde u_h-v_h\|_1+\|(u_H\cdot\nabla)e_h\|_{-1}+\|\rho_{h,H}\|_{-1}
+\|\dot e_h\|_{-1}.
\end{eqnarray*}
We first observe that applying (\ref{eq:lips1b}) we get $\|(u_H\cdot\nabla)e_h\|_{-1}\le C \|e_h\|_0$.
To bound  $\|\rho_{h,H}\|_{-1}$ we apply~(\ref{primer}).
Now, the proof concludes applying  Theorems~\ref{teo1}--\ref{teo4}
together with (\ref{cotae_t}).
\end{proof}
We state in the following theorem the results that can be obtained for the mini-element.
\begin{Theorem}
\label{teomini}
Let $(u,p)$ be the solution
of {\rm (\ref{onetwo})--(\ref{ic})}. There exists a positive
constant $C$ such that the approximations over the finer grid computed
using the mini-element, $(\tilde u_h,\tilde p_{h})$, satisfy the following bounds
for $t\in(0,T]$:
\begin{eqnarray*}
\|\tilde u(t)-u(t)\|_0&\le& C H^2+C h^2,\nonumber\\
\|\tilde u(t)-u(t)\|_1&\le&C |\log(h)|H^2+C h,\nonumber\\
\|\tilde p_h(t)-p(t)\|_{L^2/{\Bbb R}}&\le& {C}|\log(h)|H^2+C h.
\end{eqnarray*}
\end{Theorem}


\section{Semi-discretization in space. The second two-grid algorithm.}
As in the previous section we will concentrate on the approximations obtained with the Hood-Taylor mixed finite element
and $r=3$ or $r=4$ and we will state at the end of the section the results that can be obtained for the mini-element method
with a much simpler analysis.

In the second algorithm we consider, the first level, as before, is given by the
standard mixed finite-element approximation to (\ref{onetwo})--(\ref{ic}), that is,
the solution of~(\ref{ten})--(\ref{ten2}) with initial condition
$u_H(0)=\Pi_H(u_{0})$.
In the second level we solve a linearized problem on a finer grid and given $\tilde u_h(0)=\Pi_hu_0$, we compute
$\tilde u_{h}(t)\in X_{h,r}$ and $\tilde p_{h}(t)\in Q_{h,r-1}$, $t\in(0,T]$, satisfying, for all $\phi_{h} \in X_{h,r}$
and $\psi_{h} \in Q_{h,r-1}$
\begin{eqnarray}
&&(\dot  {\tilde u}_{h}, \phi_{h})
+ ( \nabla \tilde u_{h}, \nabla \phi_{h}) + b(u_{H},\tilde u_h, \phi_{h}) +b(\tilde u_h,u_H,\phi_h)
+( \nabla \tilde p_{h}, \phi_{h})  = \nonumber\\
&&\qquad\qquad\qquad\qquad\qquad\qquad\qquad\qquad\quad\quad(f, \phi_{h})+b(u_H,u_H,\phi_h),\label{second_1}
\end{eqnarray}
\begin{equation}
(\nabla \cdot \tilde u_{h}, \psi_{h})  = 0.\label{second_2}
\end{equation}
Observe that the approximation $\tilde u_H$ is the result of one step of Newton's
method for the Galerkin $(u_h,p_h)$~approximation in $X_{h,r}\times Q_{h,r-1}$
(equations~(\ref{ten})--(\ref{ten2}) with $H$ replaced by~$h$) with
$(u_H,p_H)$
as initial approximation. For this reason, in this section we study the error
$e_h=u_h-\tilde u_h$.

It is easy to obtain that
\begin{eqnarray}\label{new_e}
\dot e_h+A_h e_h+B_h(u_H,e_h)+B_h(e_h,u_H)=
\Pi_h\rho_{h,H},
\end{eqnarray}
where
$$
 \rho_{h,H} =-F(\epsilon_{h,H},\epsilon_{h,H}),
$$
where, here and in the sequel,
$$
\epsilon_{h,H}=u_H-u_h.
$$

The analysis in this section is closely related to that in
the previous section. However, some extra results are needed. We shall use
the following two bounds,
\begin{align}
\left\| \phi_h\right\|_{L^{2d/(d-1)}} &\le C\left\|\phi_h\right\|_0^{1/2}
 \bigl\|A_h^{1/2}\phi_h\bigr\|_0^{1/2},\qquad \forall \phi_h\in V_{h,r},
\label{cota_hr0}
\\
\left\| \phi_h\right\|_{L^\infty} &\le C\bigl\|A_h^{1/2}\phi_h\bigr\|_0^{1/2}
 \bigl\|A_h\phi_h\bigr\|_0^{1/2},\qquad \forall \phi_h\in V_{h,r},
\label{cota_hr}
\end{align}
which follow from Corollary~4.4 and Lemma~4.4
in~\cite{heyran0}.
Also we shall use the following two bounds
\begin{align}
\label{jul2}
\|A_h^{-1/2}B_h(\epsilon_{h,H},\epsilon_{h,H})\|_0&\le
C\bigl\|\epsilon_{h,H}\bigr\|_0^{1/2} \bigl\|\epsilon_{h,H}\bigr\|_1^{3/2},
\\
\label{jul3}
\|A_h^{-1}B_h(\epsilon_{h,H},\epsilon_{h,H})\|_0&\le
C\bigl\|\epsilon_{h,H}\bigr\|_0 \bigl\|\epsilon_{h,H}\bigr\|_1,
\end{align}
with $C$ independent of~$h$, and where here and in the sequel
$B_h(v_h,w_h)=\Pi_hF(u_h,w_h)$. Both are easily obtained by duality
arguments, the first one from~(\ref{cota_hr0}) and the second one
from~(\ref{cota_hr}).
Notice also that as a consequence of~(\ref{cota_hr0})-(\ref{cota_hr})
and~(\ref{lee1})
we have that $\bigl\|e^{-tA_h}\phi_h\bigr\|_{L^\infty}\le
Ct^{-3/4}\left\|\phi_h\right\|_0$ and $\bigl\|e^{-tA_h}A_h^{1/2}\phi_h\bigr\|_{L^{2d/(d-1)}}\le
Ct^{-3/4}\left\|\phi_h\right\|_0$ so that by using duality arguments together with these two inequalities the following
two bounds easily follow
\begin{align}
\label{jul4}
\bigl\| e^{-(t-s)A_h}B_h(\epsilon_{h,H},\epsilon_{h,H})\bigr\|_0\le &\frac{C}{(t-s)^{3/4}}
\bigl\|\epsilon_{h,H}\bigr\|_0 \bigl\|\epsilon_{h,H}\bigr\|_1.
\\
\label{jul4b}
\bigl\| e^{-(t-s)A_h}A_h^{1/2}B_h(\epsilon_{h,H},\epsilon_{h,H})\bigr\|_0\le &\frac{C}{(t-s)^{3/4}}
\bigl\|\epsilon_{h,H}\bigr\|_1^2.
\end{align}

\begin{lema}
\label{le:newton1}
There exists a positive constant $C=C(M_2)$ such that
$$
\left\| u-u_H\right\|_{L^\infty} \le CH^{1/2}.
$$
\end{lema}
\begin{proof}
We will use the
fact that, due to Lemma~4.3 and~4.4 in~\cite{heyran0}, and Corollary 4.4
in~\cite{heyran0},
\begin{equation}
\label{eq:jul1}
\left\| \nabla\Pi_Hu\right\|_{L^6}\le C\left\|A
u\right\|_0
\end{equation}
We write $u-u_H=(I-\Pi_H)u+(\Pi_Hu-u_H)$. Applying (\ref{inv}), we have
$\left\| \Pi_H u-u_H\right\|_{L^\infty}\le CH^{-3/2}\left\| \Pi_hu-u_H\right\|_{0}\le
CH^{1/2}$, where in the last inequality we have applied~(\ref{eq:error-Pi_h})
and~(\ref{gal_cuad}). On the other hand, applying~\cite[(4.43)]{heyran0}
\begin{align*}
\left\| (I-\Pi_H)u\right\|_{L^\infty} &\le C\left\| (I-\Pi_H)u\right\|_{L^6}^{1/2}
\left\|\nabla(I-\Pi_H)u\right\|_{L^6}^{1/2}\\
&\le C\left\| (I-\Pi_H)u\right\|_{1}^{1/2}(\left\|\nabla u\right\|_{L^6}+
\left\|\nabla\Pi_hu\right\|_{L^6})^{1/2}.
\end{align*}
Now,
where, in the last inequality we have applied (\ref{sob1}) and \cite[Lemma 4.4]{heyran0}.
Furthermore, applying (\ref{eq:error-Pi_h}), (\ref{sob1}) and~(\ref{eq:jul1}) the proof is
finished.
\end{proof}

\begin{lema}
\label{le:cota_pre_e1_second}
Let $(u,p)$ be the solution of {\rm (\ref{onetwo})--(\ref{ic})}. Then there
exists a positive constant $C$ such that the approximations $u_h$
and $\tilde u_{h}$ to the velocity $u$ over the fine mesh satisfy
the following bound:
\begin{equation}
\|A_h^{l/2}(u_h(t)-\tilde u_h(t))\|_0 \le C H^{7/2-l},\quad r\ge 3,\quad l=0,1,\quad t\in (0,T].
\end{equation}
\end{lema}
\begin{proof} Follow the arguments in the proof
of~Lemma~\ref{le:cota_pre_e1} to obtain~(\ref{jul5}). Now, for $l=1$
we write $s^{1/2}\left\| \Pi_h\rho_{h,H}\right\|_0\le Cs^{1/2}\left\|\epsilon_{h,H}
\right\|_{L^\infty}\left\|\epsilon_{h,H}\right\|_1$, so that applying Lemma~\ref{le:newton1}
and~(\ref{gal_cuad}) the proof of the case~$l=1$ is finished.
For $l=0$, applying~(\ref{jul2}) we have
\begin{align*}
s^{1/2}\left\|A_h^{-1/2} \Pi_h\rho_{h,H}\right\|_0&\le
Cs^{1/2}\left\|\epsilon_{h,H}\right\|_0^{1/2}\left\|\epsilon_{h,H}\right\|_1^{3/2}\\
&{}=
C\bigl(\left\|\epsilon_{h,H}\right\|_0
\left\|\epsilon_{h,H}\right\|_1\bigr)^{1/2}\bigl\|s^{1/2}\epsilon_{h,H}\bigr\|_1,
\end{align*}
so that applying~(\ref{gal_cuad}) the proof is finished.
\end{proof}

\begin{lema}
\label{le:aux2_second}
Let $(u,p)$ be the solution of {\rm (\ref{onetwo})--(\ref{ic})}. Then there
exists a positive constant $C$ such that the approximations $u_h$
and $\tilde u_{h}$ to the velocity $u$ over the fine mesh satisfy
the following bound:
\begin{eqnarray*}
\|A_h^{-1}(u_h(t)-\tilde u_h(t))\|_0 \le C H^5,\quad r\ge 3,\quad t\in (0,T].
\end{eqnarray*}
\end{lema}
\begin{proof}
Follow the arguments in the proof of Lemma~\ref{le:Cons-4.2}, but notice that now
due to the terms $(e_h\cdot \nabla) u_H$ and~$(\nabla\cdot u_H)e_H$, instead
of~(\ref{jul6}) we have
\begin{align*}
\left\| e^{-(t-s)A_h}
A_h^{-1}\Pi_h\bigl((u_H\cdot\nabla) e_h\bigr)\right\|_0\le &C \left(\frac{1}{\sqrt{t-s}}+\frac{1}{\sqrt{s}}\right)\|A_h^{-1}e_h\|_0\nonumber\\
&{}+
C\Bigl( \frac{H^3}{\sqrt{s}} \|e_h\|_1+
\frac{H^2}{\sqrt{s}} \|e_h\|_0\Bigr).
\nonumber
\end{align*}
Thus, instead of~(\ref{jul8}) we arrive at
\begin{align*}
\max_{0\le t\le T}\|y_h(t)\|_0\le &C\max_{0\le t\le T}\left\|\int_0^te^{-(t-s)A_h} A_h^{-1}\Pi_h\rho_{h,H}(s)\,ds\right\|
\nonumber\\
&{}+C \bigl(H^3 \max_{0\le t\le T}\|e_h(t)\|_1+H^2 \max_{0\le t\le T}\|e_h(t)\|_0\bigr).
\end{align*}
Thanks to Lemma~\ref{le:cota_pre_e1_second} we have that the last
two terms on the right-hand side above are bounded by~$CH^{11/2}$. For the first
one, applying first~(\ref{jul3}) and then~(\ref{gal_int}) we conclude that it is
also bounded by $CH^5$.
\end{proof}

\begin{lema}
\label{le:cota_pre_e1_second_jul}
Let $(u,p)$ be the solution of {\rm (\ref{onetwo})--(\ref{ic})}. Then there
exists a positive constant $C$ such that the approximations $u_h$
and $\tilde u_{h}$ to the velocity $u$ over the fine mesh satisfy
the following bound:
\begin{equation}
\|A_h^{-1/2}(u_h(t)-\tilde u_h(t))\|_0
\le \frac{C}{t^{1/2}} H^{5},\quad r\ge 3,\quad t\in (0,T].
\end{equation}
\end{lema}
\begin{proof} Let $y_h(t)=t^{1/2}A_h^{-1/2} e_h(t)$ and
follow the arguments in the proof
of Lemma~\ref{le:aux2}  so that instead of~(\ref{eq:antes_gro}) we now
have
\begin{eqnarray}\label{eq:antes_gro_jul}
\left\| y_h(t)\right\|_0 &\le& \int_0^t \frac{C}{\sqrt{t-s}} \left\|
A_h^{-1}s^{1/2}(B_h(e_h,u_H)+B_h(u_H,e_h)\right\|_0\,ds  \nonumber\\
&&+\biggl\| \int_0^t e^{-(t-s)A_h}
A_h^{-1/2}\Pi_hs^{1/2}\rho_{h,H}(s)\,ds \biggr\|_0,
\nonumber
\\
&&+\biggl\| \frac{1}{2}\int_0^t e^{-(t-s)A_h}
s^{-1/2}A_h^{-1/2}e_h(s)\,ds \biggr\|_0.
\end{eqnarray}
Now observe that by using $\left\| u_H(s)-u(s)\right\|_j\le CH^{3-j}/s^{1/2}$, instead
of~(\ref{jul11}) we now have
$$
 \left\|
A_h^{-1}s^{1/2}(B_h(e_h,u_H)+B_h(u_H,e_h)\right\|_0
\le C\Bigl( \left\|y_h(s)\right\|_0 + H^3\left\|e_h\right\|_1+H^2\left\|e_h\right\|_0\Big).
$$
Thus, instead of~(\ref{jul12}) we now get
\begin{align}
\max_{0\le t\le T}\left\| y_h(t)\right\|_0 \le &C
B\left(\frac{1}{2},\frac{1}{2}\right)\Bigl(\max_{0\le s\le T} s\bigl\|
A_{h}^{-1}\Pi_h \rho_{h,H} \bigr\|_0+
\max_{0\le s\le T}\bigl\|
A_{h}^{-1} e_h(s) \bigr\|_0\Bigr)
\nonumber \\
&{}+C\bigl( H^3\max_{0\le t\le T}\|e_h\|_1+H^2\max_{0\le t\le T}\|e_h\|_0\bigr).
\label{jul13}
\end{align}
Due to Lemma~\ref{le:cota_pre_e1_second} we have that the last two terms
on the right-hand side of~(\ref{jul13}) are $o(H^5)$, and due
to~Lemma~\ref{le:aux2_second} the second one is $O(H^5)$. Finally
due to~(\ref{jul3}) the first one can be bounded by
$C\bigl\|s^{1/2} \epsilon_{h,H}\bigr\|_0\bigl\|s^{1/2} \epsilon_{h,H}\bigr\|_1$, which, due to
(\ref{gal_cuad}) is also $O(H^5)$.
\end{proof}

\begin{Theorem}
\label{teo1_second}
Let $(u,p)$ be the solution of {\rm (\ref{onetwo})--(\ref{ic})}. Then there
exists a positive constant $C$ such that the approximations $u_h$
and $\tilde u_{h}$ to the velocity $u$ over the fine mesh satisfy
the following bound:
\begin{equation}\label{supercuad_tg_second}
 \qquad \|u_{h}(t)-\tilde u_{h}(t)\|_{0} \le
 \frac{C}{t}
 H^{5},\quad t\in(0,T],\quad r\ge 3.
\end{equation}
\end{Theorem}
\begin{proof}
Let $y_h(t)=te_h(t)$ and argue as in the proof of Theorem~\ref{teo1} so that
similarly to~(\ref{delteo1}) we now get
\begin{align*}
 \max_{0\le t\le T}\left\| y_h(t)\right\|_0
 \le &C \biggl(\max_{0\le
t\le T} \biggl\| \int_0^t\! e^{-A_h(t-s)}s\Pi_h\rho_{h,H}\,ds \biggr\|_0
\nonumber\\
{}&+ \max_{0\le t\le T}\biggl\|\int_0^t\!
e^{-A_h(t-s)}e_h\,ds \biggr\|_0\biggr).
\nonumber
\end{align*}
Using~(\ref{jul4}) to bound the first term on the right-hand side above,
and~(\ref{labeta}) for the second one, we get
\begin{align*}
 \max_{0\le t\le T}\left\| y_h(t)\right\|_0
 \le &C\Bigl(T^{1/4} \max_{0\le t\le T} \bigl\|t^{1/2} \epsilon_{h,H}(t)\bigr\|_0
\bigl\|t^{1/2} \epsilon_{h,H}(t)\bigr\|_1
\\
&{}+ B\left(\frac{1}{2},\frac{1}{2}\right)
\max_{0\le t\le T} \bigl\|t^{1/2} A_h^{-1/2}e_h(t)\bigr\|_0\Bigr),
\end{align*}
so that applying~(\ref{gal_cuad}) and
Lemma~\ref{le:cota_pre_e1_second_jul} the proof is finished.
\end{proof}

\begin{Theorem}
\label{teo4_second}
Let $(u,p)$ be the solution of {\rm (\ref{onetwo})--(\ref{ic})}. Then there
exists a positive constant $C$ such that the approximations $u_h$
and $\tilde u_{h}$ to the velocity $u$ over the fine mesh satisfy
the following bound:
\begin{equation}\label{supercub_tg_1_second}
 \qquad \|u_{h}(t)-\tilde u_{h}(t)\|_{1} \le
 \frac{C}{t^{(r-1)/2}}
 H^{r+1},\quad t\in(0,T],\quad r=3,4.
\end{equation}
\end{Theorem}
\begin{proof}
Let $y_h(t)=t^{(r-1)/2}A_h^{1/2}e_h(t)$ and follow the arguments in the
proof of~Lemma~\ref{le:cota_pre_e1} so that now, instead of~(\ref{delema8}) we get
\begin{align}
\max_{0\le t\le T}\left\| y_h(t)\right\|_0 \le &C\biggl(\max_{0\le t\le T}
\biggl\|\int_0^t e^{-(t-s)A_h}s^{(r-1)/2} A_h^{1/2}\Pi_h\rho_{h,H}(s)\,ds \biggr\|_0
\nonumber\\
&{}+\frac{(r-1)}{2}
\biggl\|\int_0^t e^{-(t-s)A_h}s^{(r-3)/2} A_h^{1/2}e_h(s)\,ds \biggr\|_0\biggr).
\label{jul15}
\end{align}
Applying~(\ref{jul4b}) to bound the first term on the right-hand side above
and~(\ref{labeta}) for the second one, we have
\begin{align}
\max_{0\le t\le T}\left\| y_h(t)\right\|_0 \le &C\Bigl(T^{1/4}\max_{0\le t\le T}
\bigl\|t^{1/2}\epsilon_{h,H}(t)\bigr\|_1\bigl\|t^{(r-2)/2}\epsilon_{h,H}(t)\bigr\|_1
\nonumber
\\
&{}+C
B\left(\frac{1}{2},\frac{1}{2}\right)
\max_{0\le t\le T} \bigl\|t^{(r-2)/2} e_h(t)\bigr\|_0\Bigr).
\label{jul14}
\end{align}
Due to~(\ref{gal_cuad}) the first term on the right-hand side above is
bounded by $CH^2H^{r-1}=CH^{r+1}$. For $r=4$, the second one is
bounded in~Theorem~\ref{teo1_second}. When $r=3$, we may write
$\bigl\|t^{1/2}e_h(t)\bigr\|_0=\left\|te_h(t)\right\|_0^{1/2}\left\|e_h(t)\right\|_0^{1/2}$
so that applying~Theorem~\ref{teo1_second} and Lemma~\ref{le:cota_pre_e1_second}, the
second term on the right-hand side of~(\ref{jul14}) is bounded
by~$CH^{5/2}H^{7/4}=o(H^4)$
\end{proof}

\begin{remark}\label{re:logh}\rm
For $r=3$ it is possible to prove the bound
$$
\|u_{h}(t)-\tilde u_{h}(t)\|_{1} \le
 \frac{C}{t}\left(
 H^{9/2}\left|\log(h)\right|+H^{17/4}\right),\qquad t\in(0,T],\quad r=3.
$$
To do so, apply Lemma~\ref{l4b} and~(\ref{jul2}) to bound the first term on the right-hand
side of~(\ref{jul15}) and the same bound as before for the second term.
\end{remark}
Finally, repeating (with obvious changes) the analysis in Section~\ref{se:3} for the pressure, the following result is easily proved

\begin{Theorem}
\label{teo5_second}
Let $(u,p)$ be the solution
of {\rm (\ref{onetwo})--(\ref{ic})}. There exists a positive
constant $C$ such that the approximation to $p$ over the finer grid, $\tilde p_{h}$, satisfy the following bound
for $t\in(0,T]$ and $r=3,4$:
\begin{equation}
\|\tilde p_h(t)-p_h(t)\|_{L^2/{\Bbb R}}\le  \frac{C}{t^{(r-1)/2}}
 H^{r+1},\qquad t\in(0,T].
\end{equation}
\end{Theorem}
We now summarize the main results of the section in the following theorem.

\begin{Theorem}
\label{teotodo2}
Let $(u,p)$ be the solution
of {\rm (\ref{onetwo})--(\ref{ic})}. There exists a positive
constant $C$ such that the approximations $(\tilde u_{h},\tilde p_h)$ satisfy the following bounds for $r=3,4$ and $t\in(0,T]$:
\begin{align*}
\|u(t)-\tilde u_h(t)\|_0\le& \frac{C}{t}H^5+\frac{C}{t^{(r-2)/2}}h^r.
\\
\|u(t)-\tilde u_h(t)\|_1\le&\frac{C}{t^{(r-1)/2}}H^{r+1}+\frac{C}{t^{(r-2)/2}}h^{r-1},
\\
\|\tilde p_h(t)-p(t)\|_{L^2/{\Bbb R}}\le&  \frac{C}{t^{(r-1)/2}}
 H^{r+1}+\frac{C}{t^{(r-2)/2}}h^{r-1}.
\end{align*}
\end{Theorem}
\begin{proof}
We use the decomposition $u-\tilde u=(u-u_h)+(u_h-\tilde u)$ and apply (\ref{gal_cuad}) to bound the first term and
Theorems~\ref{teo1_second} and~\ref{teo4_second} for the second. For the pressure, using the
decomposition $p-\tilde p_h=(p-p_h)+(p_h-\tilde p_h)$ and applying (\ref{eq_lap_h}) and Theorem~\ref{teo5_second} the proof is finished.
\end{proof}
Finally, with a much simpler analysis, that we do not detail here for brevity, the following result can be proved

\begin{Theorem}
\label{teomini_second}
Let $(u,p)$ be the solution
of {\rm (\ref{onetwo})--(\ref{ic})}. There exists a positive
constant $C$ such that the approximations over the finer grid computed
using the mini-element, $(\tilde u_h,\tilde p_{h})$, satisfy the following bounds
for $t\in(0,T]$:
\begin{eqnarray*}
\|\tilde u_h(t)-u(t)\|_0&\le& C H^3+C h^2,\nonumber\\
\|\tilde u_h(t)-u(t)\|_1&\le&C H^{2}+C h,\nonumber\\
\|\tilde p_h(t)-p(t)\|_{L^2/{\Bbb R}}&\le& {C}H^{2}+C h.
\end{eqnarray*}
\end{Theorem}


\section{Fully discrete case.}
In this section we consider the fully discrete case. Let us assume that we integrate in time
equations (\ref{ten}-\ref{ten2}) and (\ref{ten_p1}-\ref{ten2_p1}) for the first method
or equations (\ref{ten}-\ref{ten2}) and (\ref{second_1}-\ref{second_2}) for the second method
using the backward Euler method or the two-step backward differentiation formula (BDF). In the case of the two-step BDF method
the first step is carried out using the backward Euler method.  We will
denote by ($U_H^n,P_H^n)$ the fully discrete Galerkin approximations to the velocity and
pressure at the time level $t_n=nk$ for $0\le n\le N$
and $k=\Delta t=T/N$. We will denote by $(\widetilde U_h^n,\widetilde P_h^n)$ the fully discrete approximations
to the velocity and pressure over the finer grid at the time level $t_n$.

Let us denote by $e_H^n= u_H(t_n)-U_H^n$ and by $\tilde e_h^n=\tilde u_h(t_n)-\widetilde U_h^n$
the temporal errors in the approximations $U_H^n$ and $\widetilde U_h^n$ respectively. Let
us denote by $\pi_H^n=p_H(t_n)-P_H^n$ and by $\tilde \pi_h^n=\tilde p_h(t_n)-\widetilde P_h^n$ the
temporal errors in the approximations $P_H^n$ and $\widetilde P_H^n$ respectively. In \cite{jbj_fully} we have proved the
following error bounds. There exist  constants $C_{l_0}$ and $k_0$ such that for $k\le k_0$ the temporal
errors of the Galerkin approximation satisfy the following error bounds
$$
\|e_H^n\|_0+t_n\|A_H e_H^n\|_0\le C_{l_0}\frac{k^{l_0}}{t_n^{l_0-1}},\ \|\pi_H^n\|_{L^2(\Omega)/{\Bbb R}}\le C_{l_0}\frac{k^{l_0}}{t_n^{(2l_0-1)/2}},\ 1\le n\le N,
$$
where $l_0=1$ for the Euler method and $l_0=2$ for the two-step BDF. Let us remark that using
$\|e_H^n\|_1\le C\|A_H^{1/2} e_H^n\|_0\le C \|e_H^n\|_0^{1/2}\|A_H e_H^n\|_0^{1/2}$ error bounds in the
$H^1$ norm are also obtained in a straightforward manner.

Using the same technique developed in \cite{jbj_fully} it can also be proved that analogous error bounds hold for the approximations
over the finer grid. More precisely, for both the first and second algorithms there exist  constants $C_{l_0}$ and $k_0$ such that for $k\le k_0$ the temporal
errors of the two-grid approximation satisfy the following error bounds
\begin{eqnarray*}
\|\tilde e_h^n\|_0+t_n\|A_h \tilde e_h^n\|_0\le C_{l_0}\frac{k^{l_0}}{t_n^{l_0-1}},\ \|\tilde \pi_h^n\|_{L^2(\Omega)/{\Bbb R}}\le C_{l_0}\frac{k^{l_0}}{t_n^{(2l_0-1)/2}},\ 1\le n\le N,
\end{eqnarray*}
where $l_0=1$ for the Euler method and $l_0=2$ for the two-step BDF.

Finally, using the decompositions
\begin{eqnarray*}
u(t_n)-\widetilde U_h^n&=&(u(t_n)-\tilde u_h(t_n))+\tilde e_h^n,\nonumber\\
p(t_n)-\widetilde P_h^n&=&(p(t_n)-\tilde p_h(t_n))+\tilde \pi_h^n,
\end{eqnarray*}
the error bounds of the fully discrete approximations are obtained as the sum of the spatial errors (the errors in the semi-discrete approximations
we have already bounded in the previous sections)
plus the temporal errors.

\end{document}